\documentclass[a4paper,12pt]{amsart}

\usepackage[T1]{fontenc}
\usepackage{fullpage}
\usepackage[colorlinks]{hyperref}
\usepackage[english]{babel}
\usepackage{csquotes}
\usepackage{mathrsfs}
\usepackage{stmaryrd}
\usepackage[all,cmtip]{xy}
\usepackage{paralist}
\usepackage[shortlabels]{enumitem}
\usepackage{amssymb,amsfonts,amsxtra,amsmath,amsthm}

\theoremstyle{plain}
\newtheorem{thm}{Theorem}[section]

\newtheorem{cor}[thm]{Corollary} 
\newtheorem{prp}[thm]{Proposition} 
\newtheorem{lem}[thm]{Lemma} 
\theoremstyle{definition}
\newtheorem{dfn}[thm]{Definition}

\theoremstyle{remark} 
\newtheorem{rmk}[thm]{Remark}
\newtheorem{ntn}[thm]{Notation}

\newcommand{\NN}{\mathbb{N}}

\newcommand{\CC}{\mathbb{C}}

\newcommand{\I}{\mathcal{I}}
\newcommand{\J}{\mathcal{J}}

\newcommand{\OO}{\mathscr{O}}
\newcommand{\QQ}{\mathscr{Q}}

\newcommand{\mm}{\mathfrak{m}}
\newcommand{\pp}{\mathfrak{p}}
\newcommand{\qq}{\mathfrak{q}}

\newcommand{\reg}{\mathrm{reg}}

\newcommand{\ol}{\overline}
\newcommand{\ul}{\underline}
\newcommand{\wh}{\widehat}
\newcommand{\wt}{\widetilde}

\newcommand{\ideal}[1]{{\left\langle#1\right\rangle}}
\newcommand{\set}[1]{{\left\{#1\right\}}}
\newcommand{\xmid}{\;\middle|\;}
\newcommand{\into}{\hookrightarrow}
\newcommand{\onto}{\twoheadrightarrow}

\DeclareMathOperator{\Ass}{Ass}
\DeclareMathOperator{\Ann}{Ann}

\DeclareMathOperator{\coker}{coker}
\DeclareMathOperator{\Der}{Der}
\DeclareMathOperator{\depth}{depth}
\DeclareMathOperator{\End}{End}
\DeclareMathOperator{\Ext}{Ext}
\DeclareMathOperator{\grade}{grade}
\DeclareMathOperator{\Gdim}{G-dim}
\DeclareMathOperator{\height}{height}
\DeclareMathOperator{\Hom}{Hom}
\DeclareMathOperator{\id}{id}
\DeclareMathOperator{\Min}{Min}
\DeclareMathOperator{\pdim}{pdim}
\DeclareMathOperator{\res}{res}
\DeclareMathOperator{\Spec}{Spec}
\DeclareMathOperator{\Supp}{Supp}

\numberwithin{equation}{section}

\begin{document}

\title[Residual duality over Gorenstein rings]
{A residual duality over Gorenstein rings with application to logarithmic differential forms}

\author[M.~Schulze]{Mathias Schulze}
\address{M.~Schulze\\
Department of Mathematics\\
TU Kaiserslautern\\
67663 Kaiserslautern\\
Germany}
\email{\href{mailto:mschulze@mathematik.uni-kl.de}{mschulze@mathematik.uni-kl.de}}

\author[L.~Tozzo]{Laura Tozzo}
\address{L.~Tozzo\\
Department of Mathematics\\
TU Kaiserslautern\\
67663 Kaiserslautern\\
Germany}
\email{\href{mailto:tozzo@mathematik.uni-kl.de}{tozzo@mathematik.uni-kl.de}}

\begin{abstract}
Kyoji Saito's notion of a free divisor was generalized by the first author to reduced Gorenstein spaces and by Delphine Pol to reduced Cohen--Macaulay spaces.
Starting point is the Aleksandrov--Terao theorem:
A hypersurface is free if and only if its Jacobian ideal is maximal Cohen--Macaulay.
Pol obtains a generalized Jacobian ideal as a cokernel by dualizing Aleksandrov's multi-logarithmic residue sequence.
Notably it is essentially a suitably chosen complete intersection ideal that is used for dualizing. 
Pol shows that this generalized Jacobian ideal is maximal Cohen--Macaulay if and only if the module of Aleksandrov's multi-logarithmic differential $k$-forms has (minimal) projective dimension $k-1$, where $k$ is the codimension in a smooth ambient space.
This equivalent characterization reduces to Saito's definition of freeness in case $k=1$.
In this article we translate Pol's duality result in terms of general commutative algebra.
It yields a more conceptual proof of Pol's result and a generalization involving higher multi-logarithmic forms and generalized Jacobian modules.
\end{abstract}

\subjclass[2010]{Primary 13H10; Secondary 13C14, 32A27}

\keywords{Duality, Gorenstein, logarithmic differential form, residue, free divisor}

\maketitle
\tableofcontents

\section{Introduction}

Logarithmic differential forms along hypersurfaces and their residues were introduced by Kyoji Saito (see \cite{Sai80}).
They are part of his theory of primitive forms and period mappings where the hypersurface is the discriminant of a universal unfolding of a function with isolated critical point (see \cite{Sai81,Sai83}). 
The special case of normal crossing divisors appeared earlier in Deligne's construction of mixed Hodge structures (see \cite{Del71}).
Here the logarithmic differential $1$-forms form a locally free sheaf.
In general a divisor with this property is called a free divisor.
Further examples include plane curves (see \cite[(1.7)]{Sai80}), unitary reflection arrangements and their discriminants (see \cite[Thm.~C]{Ter83}) and discriminants of versal deformations of isolated complete intersection singularities and space curves (see \cite[(6.13)]{Loo84} and \cite{vSt95}).
Free divisors also occur as discriminants in prehomogeneous vector spaces (see \cite{GMS11}).
In case of hyperplane arrangements the study of freeness attracted a lot of attention (see \cite{Yos14}).

\smallskip

Let $D$ be a germ of reduced hypersurface in $Y\cong(\CC^n,0)$ defined by $h\in\OO_Y$.
The $\OO_Y$-modules $\Omega^q(\log D)$ of logarithmic differential $q$-forms along $D$ and the $\OO_D$-modules $\omega_D^p$ of regular meromorphic differential $p$-forms on $D$ fit into a short exact logarithmic residue sequence (see \cite[\S2]{Sai80} and \cite[\S4]{Ale88})
\[
\xymatrix{
0\ar[r] & \Omega_Y^q\ar[r] & \Omega^q(\log D)\ar[r]^-{\res_D^q} & \omega_D^{q-1}\ar[r] & 0.
}
\]
Denoting by $\nu_D\colon\wt D\to D$ the normalization of $D$, $(\nu_D)_*\OO_{\wt D}\subseteq\omega_D^0$ (see \cite[(2.8)]{Sai80}). 
For plane curves Saito showed that equality holds exactly for normal crossing curves (see \cite[(2.11)]{Sai80}).
Granger and the first author (see \cite{GS14}) generalized this fact and thus extended the L\^e--Saito Theorem (see \cite{LS84}) by an equivalent algebraic property.
They showed that $(\nu_D)_*\OO_{\wt D}=\omega_D^0$ if and only if $D$ is normal crossing in codimension one, that is, outside of an analytic subset of $Y$ of codimension at least $3$.
The proof uses the short exact sequence 
\[
\xymatrix{
0 & \J_D\ar[l] & \Theta_Y\ar[l]_-{\ideal{-,dh}} & \Der(-\log D)\ar[l] & 0\ar[l]
}
\]
obtained as the $\OO_Y$-dual of the logarithmic residue sequence.
It involves the Jacobian ideal $\J_D$ of $D$, the $\OO_Y$-module $\Theta_Y:=\Der_\CC(\OO_Y)\cong(\Omega_Y^1)^*$ of vector fields on $Y$ and its submodule $\Der(-\log D)\cong\Omega^1(\log D)^*$ of logarithmic vector fields along $D$. 
It is shown that $\omega_D^0=\J_D^*$ and that $\J_D=(\omega_D^0)^*$ if $D$ is a free divisor.
In fact freeness of $D$ is equivalent to $\J_D$ being a Cohen--Macaulay ideal by the Aleksandrov--Terao theorem (see \cite[\S2]{Ale88} and \cite[\S2]{Ter80}).

As observed by first author (see \cite{Sch16}) the inclusion $(\nu_D)_*\OO_{\wt D}\subseteq\omega_D^0$ can be seen as $(\nu_D)_*\omega_{\wt D}^0\into\omega_D^0$.
He showed that $(\nu_X)_*\omega_{\wt X}^0=\omega_X^0$ is equivalent to $X$ being normal crossing in codimension one for reduced equidimensional spaces $X$ which are free in codimension one.
Here freeness means Gorenstein with Cohen--Macaulay $\omega$-Jacobian ideal.
As the latter coincides with the Jacobian ideal for complete intersections (see \cite[Prop.~1]{Pie79}), this generalizes the classical freeness of divisors which holds true in codimension one.

\smallskip

Multi-logarithmic differential forms generalize Saito's logarithmic differential forms replacing hypersurfaces $D\subseteq Y$ by subspaces $X\subseteq Y$ of codimension $k\ge2$.
They were first introduced with meromorphic poles along reduced complete intersections by Aleksandrov and Tsikh (see \cite{AT01,AT08}), later with simple poles by Aleksandrov (see \cite[\S3]{Ale12}) and recently along reduced Cohen--Macaulay and reduced equidimensional spaces by Aleksandrov (see \cite[\S10]{Ale14}) and by Pol (see \cite[\S4.1]{Pol16}).
The precise relation of the forms with simple and meromorphic poles was clarified by Pol (see \cite[Prop.~3.1.33]{Pol16}).
Here we consider only multi-logarithmic forms with simple poles.

The $\OO_Y$-modules $\Omega^q(\log X/C)$ of multi-logarithmic $q$-forms on $Y$ along $X$ depend on the choice of divisors $D_1,\dots,D_k$ defining a reduced complete intersection $C=D_1\cap\dots\cap D_k\subseteq Y$ such that $X\subseteq C$.
Consider the divisor $D=D_1\cup\dots\cup D_k$ defined by $h=h_1\cdots h_k\in\OO_Y$.
Due to Aleksandrov and Pol there is a multi-logarithmic residue sequence
\begin{equation}\label{126}
\xymatrix{
0\ar[r] & \Sigma\Omega_Y^q\ar[r] & \Omega^q(\log X/C)\ar[r]^-{\res_{X/C}^q} & \omega_X^{q-k}\ar[r] & 0
}
\end{equation}
where $\Sigma=\I_C(D)$ is obtained from the ideal $\I_C$ of $C\subseteq Y$ and $\omega_X^p$ is the $\OO_X$-module of regular meromorphic $p$-forms on $X$ (see \cite[\S10]{Ale14} and \cite[\S4.1.3]{Pol16}).
Pol introduced an $\OO_Y$-module $\Der^k(-\log X/C)$ of logarithmic $k$-vector fields on $Y$ along $X$ and a kind of Jacobian ideal $\J_{X/C}$ of $X$ that fit into the short exact sequence dual to \eqref{126} for $q=k$ 
\begin{equation}\label{128}
\xymatrix{
0 & \J_{X/C}\ar[l] & \Theta_Y^k\ar[l]_-{\ideal{-,\alpha_X}} & \Der^k(-\log X/C)\ar[l] & 0\ar[l]
}
\end{equation}
where $\Theta_Y^q=\bigwedge_{\OO_Y}^q\Theta_Y$ and $\begin{bmatrix}\alpha_X\\h_1,\dots,h_k\end{bmatrix}\in\omega_X^0$ is a fundamental form of $X$ (see \cite[\S4.2.2-3]{Pol16}).
Notably the duality applied here is $-^\Sigma=\Hom_{\OO_Y}(-,\Sigma)$.
Pol showed that Cohen--Macaulayness of $\J_{X/C}$ serves as a further generalization of freeness.
In fact the property is independent of $C$ (see \cite[Prop.~4.2.21]{Pol16}) and $\J_{X/C}$ coincides with the $\omega$-Jacobian ideal in case $X$ is Gorenstein (see \cite[\S4.2.5]{Pol16}).
By relating $\Sigma$- and $\OO_Y$-duality Pol established the following major result (see \cite[Thm.~4.2.22]{Pol16} or \cite{Pol15}).
In particular it generalizes Saito's original definition of freeness to the case $k>1$.


\begin{thm}[Pol]\label{111}
Let $X\subseteq C\subseteq Y\cong(\CC^n,0)$ where $X$ is a reduced Cohen--Macaulay germ and $C$ a complete intersection germ, both of codimension $k\ge1$ in $Y$.
Then
\[
\pdim(\Omega^k(\log X/C))\ge k-1
\]
with equality equivalent to freeness of $X$.
\end{thm}


In \S\ref{11} we pursue the main objective of this article: a translation of Theorem~\ref{111} in terms of general commutative algebra.
The role of $\OO_Y\onto\OO_C=\OO_Y/\I_C$ is played by a map of Gorenstein rings $R\to\ol R=R/I$ of codimension $k\ge2$.
For dualizing we use
\[
-^I=\Hom_R(-,I),\quad -^\vee=\Hom_R(-,\omega_R),\quad-^{\ul\vee}=\Hom_{\ol R}(-,\ol\omega_R)
\]
where $\omega_R$ is a canonical module for $R$ and $\ol\omega_R=\ol R\otimes_R\omega_R$, which is a canonical module for $\ol R$ due to the Gorenstein hypothesis (see Notation~\ref{99}).
Modelled after the multi-logarithmic residue sequence~\eqref{126} along $X=C$ we define an $I$-free approximation of a finitely generated $R$-module $M$ as a short exact sequence
\[
\xymatrix{
0\ar[r] & IF\ar[r]^-\iota & M\ar[r] & W\ar[r] & 0
}
\]
where $F$ is free and $W$ is an $\ol R$-module.
More precisely $M$ plays the role of $\Omega^q(\log X/C)(-D)$ which, as opposed to $\Omega^q(\log X/C)$, is independent of the choice of $D$.
The $I$-dual sequence
\[
\xymatrix{
0 & V\ar[l] & F^\vee\ar[l]_-\alpha & M^I\ar[l]_-\lambda & 0\ar[l]
}
\]
plays the role of the $\Sigma$-dual sequence~\eqref{128} for $X=C$.
In Proposition~\ref{58} we show that $M$ is $I$-reflexive if and only if $W$ is the $\ol R$-dual of $V$.
Our main result is


\begin{thm}\label{60}
Let $R$ be a Gorenstein local ring and let $I$ be an ideal of $R$ of height $k\ge2$ such that $\ol R=R/I$ is Gorenstein.
Consider an $I$-free approximation
\[
\xymatrix{
0\ar[r] & IF\ar[r]^-\iota & M\ar[r]^-\rho & W\ar[r] & 0
}
\]
of an $I$-reflexive finitely generated $R$-module $M$ with $W\ne0$ and the corresponding $I$-dual 
\[
\xymatrix{
0 & V\ar[l] & F^\vee\ar[l]_-\alpha & M^I\ar[l]_-\lambda & 0.\ar[l]
}
\]
Then $W=V^{\ol\vee}$ and $V$ is a maximal Cohen--Macaulay $\ol R$-module if and only if $\Gdim(M)\le k-1$.
In this latter case $V=W^{\ol\vee}$ is ($\ol\omega_R$-)reflexive.
Unless $\ol\alpha:=\ol R\otimes\alpha$ is injective, $\Gdim(M)\ge k-1$.
\end{thm}


Due to the Gorenstein hypothesis, Theorem~\ref{60} applies to the complete intersection ring $\ol R=\OO_C$, but in general not to $\ol R=\OO_X$.
In \S\ref{65} we describe a construction to restrict the support of an $I$-free approximation to the locus defined by an ideal $J\unlhd R$ with $I\subseteq J$. 
Lemma~\ref{76} shows that it is made in a way such that the multi-logarithmic residue sequence along $X$ is obtained from that along $C$ by restricting with $J=\I_X$.
Corollary~\ref{61} extends Theorem~\ref{60} to this generalized setup.

\smallskip 

In \S\ref{51} we apply our results to multi-logarithmic forms.
We define $\OO_Y$-submodules $\Der^q(-\log X)\subseteq\Theta_Y^q$ of logarithmic $q$-vector fields on $Y$ along $X$ independent of $C$ and show that $\Der^k(-\log X)=\Der^k(-\log X/C)$.
We further define Jacobian $\OO_X$-modules $\J_X^{n-q}\subseteq\OO_X\otimes_{\OO_Y}\Theta_Y^{q-k}$ of $X$ independent of $C$ and $Y$ such that $\J_X^{\dim X}=\J_{X/C}$.
The $\Sigma$-dual of the multi-logarithmic residue sequence reads
\[
\xymatrix{
0 & \J_X^{n-q}\ar[l] & \Theta_Y^q\ar[l]_-{\alpha^X} & \Der^q(-\log X)\ar[l] & 0\ar[l]
}
\]
where $\alpha^X$ is contraction by $\alpha_X$.
As a consequence of Corollary~\ref{61} we obtain the following result which is due to Pol in case $q=k$ (see \cite[Prop.~4.2.17, Thm.~4.2.22]{Pol16}).


\begin{thm}\label{0}
Let $X\subseteq C\subseteq Y\cong(\CC^n,0)$ where $X$ is a reduced Cohen--Macaulay germ and $C$ a complete intersection germ, both of codimension $k\ge2$ in $Y$.
For $k\le q<n$, $\omega_X^{q-k}=\Hom_{\OO_X}(\J_X^{n-q},\omega_X)$ where $\omega_X=\Hom_{\OO_C}(\OO_X,\OO_C)(D)$ and $\pdim(\Omega^q(\log X/C))\ge k-1$.
Equality holds if and only if $\J_X^{n-q}$ is maximal Cohen--Macaulay.
In this latter case $\J_X^{n-q}=\Hom_{\OO_X}(\omega_X^{q-k},\omega_X)$ is $\omega_X$-reflexive.
\end{thm}


The analogy with the hypersurface case (see \cite[(1.8)]{Sai80}) now raises the question whether $\J_X^{n-q}$ being maximal Cohen--Macaulay for $q=k$ implies the same for all $q>k$.
An explicit description of the Jacobian modules is given in Remark~\ref{106}.


\subsection*{Acknowledgments}
We thank Delphine Pol and the anonymous referee for helpful comments.

\section{Residual duality over Gorenstein rings}\label{11}

For this section we fix a Cohen--Macaulay local ring $R$ with $n:=\dim(R)$ and an ideal $I\unlhd R$ with $k:=\height(I)\ge2$ defining a Cohen--Macaulay factor ring $\ol R:=R/I$.
These fit into a short exact sequence 
\begin{equation}\label{4}
\xymatrix{
0\ar[r] & I\ar[r] & R\ar[r]^\pi & \ol R\ar[r] & 0.
}
\end{equation}

Note that (see \cite[Thm.~2.1.2.(b), Cor.~2.1.4]{BH93})
\[
n-\dim(\ol R)=\grade(I)=\height(I)=k\ge2.
\]
In particular $I$ is a regular ideal of $R$ and hence any $\ol R$-module is $R$-torsion.

We assume further that $R$ admits a canonical module $\omega_R$.
Then also $\ol R$ admits a canonical module $\omega_{\ol R}$ (see \cite[Thm.~3.3.7]{BH93}).


\begin{ntn}\label{99}
Abbreviating $\ol\omega_R:=\ol R\otimes_R\omega_R$ we deal with the following functors
\begin{alignat*}{2}
-^*&:=\Hom_R(-,R),&\quad-^\vee&:=\Hom_R(-,\omega_R),\\
-^I&:=\Hom_R(-,I\omega_R),&-^{\ol\vee}&:=\Hom_R(-,\ol\omega_R).
\end{alignat*}
In general $\ol\omega_R\not\cong\omega_{\ol R}$ and $-^{\ol\vee}$ is not the duality of $\ol R$-modules.
For an $\ol R$-module $N$, $N^*=\Hom_{\ol R}(N,\ol R)$ but $N^\vee$ means either $\Hom_R(N,\omega_R)$ or $\Hom_{\ol R}(N,\omega_{\ol R})$, depending on the context.
For $R$-modules $M$ and $N$, we denote the canonical evaluation map by
\[
\delta_{M,N}\colon M\to\Hom_R(\Hom_R(M,N),N),\quad m\mapsto(\varphi\mapsto\varphi(m)).
\]
Whenever applicable we use an analogous notation for $\ol R$-modules.
We denote canonical isomorphisms as equalities.
\end{ntn}


\begin{lem}\label{7}
Let $N$ be an $\ol R$-module.
Then $\Ext_R^i(N,\omega_R)=0$ for $i<k$ and $N^I=0$.
\end{lem}

\begin{proof}
The first vanishing is due to Ischebeck's Lemma (see \cite[Satz 1.9]{HK71}), the second holds because $\omega_R$ and hence $I\omega_R$ is torsion free (see \cite[Thm.~2.1.2.(c)]{BH93}) whereas $N$ is torsion.
\end{proof}

\subsection{\texorpdfstring{$I$}{I}-duality and \texorpdfstring{$I$}{I}-free approximation}

\begin{lem}\label{2}
There is a canonical identification $\omega_R=I^I$ and a canonical inclusion $I\into\omega_R^I$.
They combine to the map $\delta_{I,I\omega_R}\colon I\to I^{II}$ which is an isomorphism if $R$ is Gorenstein.
\end{lem}

\begin{proof}
Applying $-^\vee$ to \eqref{4} and $\Hom_R(I,-)$ to $I\omega_R\into\omega_R$ 
yields an exact sequence with a commutative triangle
\begin{equation}\label{19}
\xymatrix{
\Ext^1_R(\ol R,\omega_R) & I^\vee\ar[l] & \omega_R\ar[dl]^-\mu\ar[l] & \ol R^\vee\ar[l] & 0\ar[l]\\
&I^I.\ar@{^(->}[u]
}
\end{equation}
The diagonal map sends $\varepsilon\in\omega_R$ to the multiplication map $\mu(\varepsilon)\colon I\to I\omega_R$, $x\mapsto x\cdot\varepsilon$.
With Lemma~\ref{7} it follows that $\omega_R=I^\vee=I^I$.

There is an isomorphism $R\cong\End_R(\omega_R)$ sending each element to the corresponding multiplication map (see \cite[Thm.~3.3.4.(d))]{BH93}).
Applying $\Hom_R(\omega_R,-)$ to $I\omega_R\into\omega_R$ yields a commutative square
\begin{equation}\label{14}
\xymatrix{
R\ar[r]_-\cong & \End_R(\omega_R)\\
I\ar@{^(->}[u]\ar[r]^{\delta'} & \omega_R^I.\ar@{^(->}[u]
}
\end{equation}
If $R$ is Gorenstein, then $\omega_R^I=\Hom_R(R,I)=I$ and $\delta'$ is an isomorphism.

Combined with the above identification $\omega_R=I^I$, $\delta'$ defines a map $\delta\colon I\to I^{II}$.
Since
\[
\delta(x)(\mu(\varepsilon))=\delta'(x)(\varepsilon)=x\cdot\varepsilon=\mu(\varepsilon)(x)=\delta_{I,I\omega_R}(x)(\mu(\varepsilon))
\]
for all $x\in I$ and $\varepsilon\in\omega_R$, in fact $\delta=\delta_{I,I\omega_R}$.
\end{proof}


\begin{dfn}
If $F$ is a free $R$-module, then we call $IF=I\otimes_RF$ an \emph{$I$-free module}.
An $R$-module $M$ is called \emph{$I$-reflexive} if $\delta_{M,I\omega_R}\colon M\to M^{II}$ is an isomorphism.
\end{dfn}


\begin{prp}\label{16}
Let $F$ be a free $R$-module $F$.
Then $F^\vee=(IF)^I$ by restriction.
The adjunction map $IF\to F^{\vee I}$ is induced by the isomorphism $\delta_{F,\omega_R}$ and identifies with $\delta_{IF,I\omega_R}$.
In case $R$ is Gorenstein, $IF$ is $I$-reflexive.
\end{prp}

\begin{proof}
Applying $\Hom_R(F,-)$ to $\mu$ in \eqref{19} yields $F^\vee=(IF)^I$ by Hom-tensor adjunction. 

Applying $F\otimes_R-$ to \eqref{14} yields a commutative square
\[
\xymatrix{
F\ar[r]^-{\delta_{F,\omega_R}}_-\cong & F^{\vee\vee}\\
IF\ar@{^(->}[u]\ar[r] & F^{\vee I}\ar@{^(->}[u]
}
\]
where the bottom row is adjunction.
In fact, using Lemma~\ref{2},
\begin{align*}
IF=I\otimes_RF\to F\otimes_R\omega_R^I&=F\otimes_R\Hom_R(\omega_R,I\omega_R)\\
&=\Hom_R(F\otimes_R\omega_R,I\omega_R)\\
&=\Hom_R(F\otimes_R\Hom_R(R,\omega_R),I\omega_R)\\
&=\Hom_R(\Hom_R(F\otimes_RR,\omega_R),I\omega_R)\\
&=\Hom_R(\Hom_R(F,\omega_R),I\omega_R)=F^{\vee I},\\
x\cdot e &\mapsto (\psi\mapsto x\cdot\psi(e)).
\end{align*}
Identifying $F^\vee=(IF)^I$ using Lemma~\ref{2} yields with the map $\mu$ in diagram~\eqref{19}
\[
\varepsilon=\psi(e)\leftrightarrow\mu(\varepsilon)\implies 
x\cdot\psi(e)=x\cdot\varepsilon=\mu(\varepsilon)(x).
\]
Adjunction thus becomes identified with $\delta_{IF,I\omega_R}$.
The last claim is due to Lemma~\ref{2}.
\end{proof}


\begin{dfn}\label{8}
Let $M$ be a finitely generated $R$-module.
We call a short exact sequence
\begin{equation}\label{5}
\xymatrix{
0\ar[r] & IF\ar[r]^-\iota & M\ar[r]^-\rho & W\ar[r] & 0
}
\end{equation}
where $F$ is free and $IW=0$ an \emph{$I$-free approximation of $M$} with \emph{support} $\Supp(W)$.
We consider $W$ as an $\ol R$-module.
The inclusion map $\iota\colon IF\into F=M$ defines the \emph{trivial $I$-free approximation}
\[
\xymatrix{
0\ar[r] & IF\ar[r] & F\ar[r] & F/IF\ar[r] & 0.
}
\]
A \emph{morphism of $I$-free approximations} is a morphism of short exact sequences. 
\end{dfn}


\begin{lem}\label{25}
For any $I$-free approximation \eqref{5}, $\iota$ fits into a unique commutative triangle
\begin{equation}\label{6}
\xymatrix{
& F \\
IF\ar@{^(->}[ur]\ar@{^(->}[r]^-\iota & M.\ar[u]_-\kappa
}	
\end{equation}
If $\iota^{-1}$ denotes the choice of any preimage under $\iota$, then $\kappa(m)=\iota^{-1}(xm)/x$ for any $x\in I\cap R^{\reg}$.
If $M$ is maximal Cohen--Macaulay, then $\kappa$ is surjective.
In particular, \eqref{5} becomes trivial if in addition $\kappa$ injective.
\end{lem}

\begin{proof}
Applying $\Hom_R(-,F)$ to \eqref{5} yields
\[
\xymatrix{
\Ext^1_R(W,F) & \ar[l]\Hom_R(IF,F) & \ar[l]_-{\iota^*}\Hom_R(M,F) & \ar[l]\Hom_R(W,F) & \ar[l]0.
}
\]
By Ischebeck's Lemma (see \cite[Satz~1.9]{HK71}), $\Ext^1_R(W,F)=0=\Hom_R(W,F)$ making $\iota^*$ an isomorphism.
Then $\kappa$ is the preimage of the canonical inclusion $IF\into F$ under $\iota^*$.
The formula for $\kappa$ follows immediately.

Since $\coker(\kappa)$ is a homomorphic image of $F/IF$, $\dim(\coker(\kappa))\le n-k\le n-2$.
If $M$ is maximal Cohen--Macaulay, then $\depth(\coker(\kappa))\ge n-1$ by the Depth Lemma (see \cite[Prop.~1.2.9]{BH93}).
This forces $\coker(\kappa)=0$ (see \cite[Prop.~1.2.13]{BH93}) and makes $\kappa$ surjective.
\end{proof}


By functoriality of the cokernel, any $\varphi\in F^\vee$ gives rise to a commutative diagram
\begin{equation}\label{17}
\xymatrix{
0\ar[r] & I\omega_R\ar[r] & \omega_R\ar[r]^-{\pi_\omega} & \ol\omega_R\ar[r] & 0\\
& & F\ar[u]_-\varphi\\
0\ar[r] & IF\ar[uu]_-{\varphi\vert_{IF}}\ar@{^(->}[ur]\ar[r]^-\iota & M\ar[u]_-\kappa\ar[r]^-\rho & W \ar@{-->}[uu]_-{\ol\varphi}\ar[r] & 0
}	
\end{equation}
with top exact row induced by \eqref{4} and bottom row \eqref{5}.
This defines a map
\begin{equation}\label{26}
\xymatrix@R=0em{
W^{\ol\vee} & F^\vee\ar[l],\\
\ol\varphi & \varphi.\ar@{|->}[l]
}
\end{equation}


Applying $\Hom_R(F,-)$ to the upper row of \eqref{17} yields a short exact sequence
\begin{equation}\label{3}
\xymatrix{
0\ar[r] & F^I\ar[r] & F^\vee\ar[r] & F^{\ol\vee}\ar[r] & 0.
}
\end{equation}
By Lemma~\ref{7} applying $-^I$ to \eqref{5} and \eqref{6} yields the exact diagonal sequence and the triangle of inclusions with vertex $F^I$ in the following commutative diagram.
\begin{equation}\label{24}
\xymatrix{
0 & V\ar[l]\ar@{=}[d] & F^{\ol\vee}\ar[l]_-{\ol\alpha} & M^I/F^I\ar[l]_-{\ol\lambda} & 0\ar[l]\ar[dl]\\
0 & V\ar[l]\ar@{^(->}[dd] & F^\vee\ar[l]_-{\alpha}\ar@{->>}[u] & M^I\ar[l]_-{\lambda}\ar_-{\iota^I}[dl]\ar@{->>}[u] & 0\ar[l]\\
& & (IF)^I\ar@{=}[u]\ar[dl] & F^I\ar@{_(->}[l]\ar@{^(->}[u]_-{\kappa^I}\\
& \Ext^1_R(W,I\omega_R)
}
\end{equation}
By Proposition~\ref{16}, the identification $F^\vee=(IF)^I$ in diagram~\eqref{24} is given by
\[
\varphi\leftrightarrow\varphi\vert_{IF}=\varphi\circ\kappa\circ\iota
\]
in diagram~\eqref{17}.
It defines the map $\lambda$ with cokernel $\alpha$.
For $\psi\in M^I$, $\lambda(\psi)$ is defined by
\[
\lambda(\psi)\vert_{IF}=\psi\circ\iota.
\]
With $\Ext^1_R(W,I\omega_R)$ also $V$ is an $\ol R$-module.
Using \eqref{3} the Snake Lemma yields the short exact upper row of \eqref{24}. 
By Lemma~\ref{7} the commutative square $\Hom_R(IF\into M,I\omega_R\into\omega_R)$ reads
\[
\xymatrix{
(IF)^I\ar@{^(->}[d] & M^I\ar[l]_-{\iota^I}\ar@{^(->}[d]\\
(IF)^\vee & M^\vee.\ar[l]_-{\iota^\vee}^-\cong
}
\]
This allows one to check equalities of maps $M\to\omega_R$ after precomposing with $\iota$.
It follows that
\begin{equation}\label{105}
\varphi\circ\kappa\in M^I\iff\varphi\in\lambda(M^I)\implies\varphi=\lambda(\varphi\circ\kappa)
\end{equation}
for any $\varphi\in F^\vee$.


\begin{dfn}\label{9}
We call the middle row 
\begin{equation}\label{37}
\xymatrix{
0 & V\ar[l] & F^\vee\ar[l]_-{\alpha} & M^I\ar[l]_-{\lambda} & 0\ar[l]
}
\end{equation}
of diagram~\eqref{24} the \emph{$I$-dual} of the $I$-free approximation \eqref{5}.
We set
\begin{equation}\label{44}
W':=\Ext_R^1(V,I\omega_R).
\end{equation}
\end{dfn}


\begin{lem}\label{36}
For any $I$-free approximation \eqref{5} the map \eqref{26} factors through the map $\alpha$ in \eqref{24} defining an inclusion $\nu\colon V\to W^{\ol\vee}$, that is, 
\[
\xymatrix@R=0em{
W^{\ol\vee} & V\ar@{_{(}->}[l]_-{\nu} & F^\vee\ar@{->>}[l]_-{\alpha},\\
\ol\varphi && \varphi.\ar@{|->}[ll]
}
\]
\end{lem}

\begin{proof}
By diagrams~\eqref{17} and \eqref{24}, equivalence~\eqref{105} and exactness properties of $\Hom$,
\[
\ol\varphi=0\iff\ol\varphi\circ\rho=0\iff\varphi\circ\kappa\in M^I\iff\varphi\in\lambda(M^I)\iff\alpha(\varphi)=0.\qedhere
\]
\end{proof}


\begin{rmk}
By Lemma~\ref{7} applying $\Hom_R(W,-)$ to the upper row of diagram~\eqref{17} yields
\[
W^{\ol\vee}=\coker\Hom_R(W,\pi_\omega)\cong\Ext_R^1(W,I\omega_R).
\]
The inclusion of $V$ in the latter in diagram~\eqref{24} uses $\coker \iota^I\into\Ext_R^1(W,I\omega_R)$.
The relation with the inclusion $\nu$ in Lemma~\ref{36} is clarified by the double complex obtained by applying $\Hom_R(-,-)$ to $\eqref{5}$ and the upper row of \eqref{17}.
By Lemma~\ref{7} it expands to a commutative diagram with exact rows and columns
\[
\xymatrix{
& 0\ar[d] & 0\ar[d]\\
\Ext_R^1(W,I\omega_R)\ar[d] & (IF)^I\ar[l]\ar[d] & M^I\ar[l]_-{\iota^I}\ar[d] & 0\ar[l]\ar[d]\\
0 & (IF)^\vee\ar[l]\ar[d] & M^\vee\ar[l]_-{\iota^\vee}\ar[d] & 0\ar[l]\ar[d]^-{\Hom_R(W,\pi_\omega)}\\
& (IF)^{\ol\vee} & M^{\ol\vee}\ar[l] & W^{\ol\vee}\ar[l]\ar[d]_-\cong & 0\ar[l]\\
&&&\Ext_R^1(W,I\omega_R).
}
\]
An element $\alpha(\varphi)\in V$ with $\varphi\in F^\vee$ maps to $\varphi\vert_{IF}\in(IF)^I$, to $\varphi\circ\kappa\in M^\vee$ and to $\ol\varphi\in W^{\ol\vee}$.
\end{rmk}

\subsection{\texorpdfstring{$I$}{I}-reflexivity over Gorenstein rings}

In this subsection we assume that $R$ is Gorenstein and study $I$-reflexivity of modules $M$ in terms of an $I$-free approximation~\eqref{5}.
With the Gorenstein hypothesis $F^\vee$ is free and hence
\begin{equation}\label{47}
\Ext^1_R(F^\vee,-)=0.
\end{equation}


\begin{prp}\label{33}
Assume that $R$ is Gorenstein.
For any $I$-free approximation \eqref{5} and $W'$ as in \eqref{44} there is a commutative square
\[
\xymatrix{
M\ar[d]^-{\delta_{M,I\omega_R}}\ar@{->>}[r]^-\rho & W\ar[d]^-{\ol\delta}\\
M^{II}\ar@{->>}[r]^-{\rho'} & W'
}
\]
and $\ol\delta$ is an isomorphism if and only if $M$ is $I$-reflexive.
\end{prp}

\begin{proof}
Consider the following commutative diagram whose rows are \eqref{5} and obtained by applying $-^I$ to the triangle with vertex $F^\vee$ in diagram~\eqref{24}.
\begin{equation}\label{34}
\xymatrix{
&&F\ar@/^12.0pc/[ddd]_-\cong^-{\delta_{F,\omega_R}}\\
0\ar[r] & IF\ar@{^(->}[ur]\ar[d]^-{\delta_{IF,I\omega_R}}_-\cong\ar[r]^-\iota & M\ar[u]_-\kappa\ar[d]^-{\delta_{M,I\omega_R}}\ar[r]^-\rho & W\ar[r]\ar@{-->}[d]^-{\ol\delta} & 0\\
0\ar[r] & (IF)^{II}\ar@{=}[d]\ar[r]^-{\iota^{II}} & M^{II}\ar[r]^-{\rho'} & W'\ar[r] & 0\\
& F^{\vee I}\ar@{^(->}[r]\ar[ur]^-{\lambda^I} & F^{\vee\vee}.
}
\end{equation}
The latter is a short exact sequence by Lemma~\ref{7} and \eqref{47}.
The commutative squares in diagram~\eqref{34} are due to functoriality of $\delta$ and the cokernel.
The claimed equivalence then follows from the Snake Lemma.
Proposition~\ref{16} yields the part of diagram~\eqref{34} involving $\delta_{F,\omega_R}$.
This part is just added for clarification but not needed for the proof.
\end{proof}


\begin{lem}\label{35}
Assume that $R$ is Gorenstein and consider an $I$-free approximation \eqref{5}.
Then the maps $\nu$ from Lemma~\ref{36} and $\ol\delta$ from Proposition~\ref{33} fit into a commutative square
\[
\xymatrix{
W\ar[d]^-{\ol\delta}\ar[r]^-{\delta_{W,\ol\omega_R}} & W^{\ol\vee\ol\vee}\ar[d]^-{\nu^{\ol\vee}}\\
W' & V^{\ol\vee}\ar[l]_-\xi^-\cong.
}	
\]
\end{lem}

\begin{proof}
Consider the double complex obtained by applying $\Hom_R(-,-)$ to the middle and top rows of diagrams~\eqref{24} and \eqref{17}.
By Lemma~\ref{7} and \eqref{47} it expands to a commutative diagram with exact rows and columns 
\[
\xymatrix{
&& 0\ar[d] & 0\ar[d]\\
& 0\ar[r]\ar[d] & F^{\vee I}\ar[r]^-{\lambda^I}\ar[d] & M^{II}\ar[r]^-{\rho'}\ar[d] & W'\ar[r]\ar[d] & 0\\
& 0\ar[r]\ar[d] & F^{\vee\vee}\ar[r]^-{\lambda^\vee}\ar[d]^-{(\pi_\omega)_*} & M^{I\vee}\ar[r]\ar[d] & 0\\
0 \ar[r] & V^{\ol\vee}\ar[r]^-{\alpha^{\ol\vee}} & F^{\vee\ol\vee}\ar[r]^-{\lambda^{\ol\vee}}\ar[d] & M^{I\ol\vee}\\
&& 0.
}
\]
The Snake Lemma yields an isomorphism $\xi\colon V^{\ol\vee}\to W'$.
Attaching the square of Proposition~\ref{33}, the relation $\ol\delta(w)=\xi(\wt\psi)$ is given by the diagram chase
\[
\xymatrix{
&& m\ar@{|->}[d]\ar@{|->}[r] & w\ar@{|->}[d]\\
&& \delta_{M,I\omega_R}(m)\ar@{|->}[d]\ar@{|->}[r] &  \ol\delta(w)\\
& \psi\ar@{|->}[d]\ar@{|->}[r] & \psi\circ\lambda=\delta_{M,I\omega_R}(m)\\
\wt\psi\ar@{|->}[r] & \wt\psi\circ\alpha=\pi_\omega\circ\psi.
}
\]
Using implication~\eqref{105}, diagram~\eqref{17} and Lemma~\ref{36}, one deduces that, with $x\in I\cap R^\reg$ and $v=\alpha(\varphi)$,
\begin{align*}
x\varphi\circ\kappa\in M^I
\implies x\varphi
&=\lambda(x\varphi\circ\kappa)\\
\implies x\psi(\varphi)
&=\psi(x\varphi)
=(\psi\circ\lambda)(x\varphi\circ\kappa)
=\delta_{M,I\omega_R}(m)(x\varphi\circ\kappa)
=x(\varphi\circ\kappa)(m)\\
\implies\psi(\varphi)
&=(\varphi\circ\kappa)(m)\\
\implies\wt\psi(v)
&=(\wt\psi\circ\alpha)(\varphi)
=(\pi_\omega\circ\psi)(\varphi)
=(\pi_\omega\circ\varphi\circ\kappa)(m)
=\ol\varphi(w)\\
&=(\nu\circ\alpha)(\varphi)(w)
=\nu(\alpha(\varphi))(w)
=\nu(v)(w)\\
&=\delta_{W,\ol\omega_R}(w)(\nu(v))
=\nu^{\ol\vee}(\delta_{W,\ol\omega_R}(w))(v)
=(\nu^{\ol\vee}\circ\delta_{W,\ol\omega_R})(w)(v)\\
\implies\wt\psi
&=(\nu^{\ol\vee}\circ\delta_{W,\ol\omega_R})(w)\\
\implies\ol\delta(w)
&=\xi(\wt\psi)
=(\xi\circ\nu^{\ol\vee}\circ\delta_{W,\ol\omega_R})(w)\\
\implies\ol\delta
&=\xi\circ\nu^{\ol\vee}\circ\delta_{W,\ol\omega_R}.\qedhere
\end{align*}
\end{proof}


\begin{prp}\label{58}
Assume that $R$ is Gorenstein and consider an $I$-free approximation \eqref{5}.
Then $M$ is $I$-reflexive if and only if the map $\nu^{\ol\vee}\circ\delta_{W,\ol\omega_R}$ with $\nu$ from Lemma~\ref{36} identifies $W=V^{\ol\vee}$.
\end{prp}

\begin{proof}
The claim follows from Proposition~\ref{33} and Lemma~\ref{35}.
\end{proof}


\begin{lem}\label{50}
Assume that $R$ is Gorenstein and consider an $I$-free approximation \eqref{5}.
Then the map $\nu$ from Lemma~\ref{36} fits into a commutative diagram
\[
\xymatrix@C=6em@R=3em{
W^{\ol\vee} & V\ar[d]^-{\delta_{V,\ol\omega_R}}\ar[l]_-{\nu} & \\
W^{\ol\vee\ol\vee\ol\vee}\ar[u]^-{\delta_{W,\ol\omega_R}^{\ol\vee}} & V^{\ol\vee\ol\vee}\ar[l]^-{\nu^{\ol\vee\ol\vee}}\ar[ul]_-{(\nu^{\ol\vee}\circ\delta_{W,\ol\omega_R})^{\ol\vee}}.
}
\]
\end{lem}

\begin{proof}
For any $v\in V$ and $w\in W$ we have
\begin{align*}
(\delta_{W,\ol\omega_R}^{\ol\vee}\circ\nu^{\ol\vee\ol\vee}\circ\delta_{V,\ol\omega_R})(v)(w)
&=\delta_{W,\ol\omega_R}^{\ol\vee}(\nu^{\ol\vee\ol\vee}(\delta_{V,\ol\omega_R}(v)))(w)
=\delta_{W,\ol\omega_R}^{\ol\vee}(\delta_{V,\ol\omega_R}(v)\circ\nu^{\ol\vee})(w)\\
&=(\delta_{V,\ol\omega_R}(v)\circ\nu^{\ol\vee})(\delta_{W,\ol\omega_R}(w))
=\delta_{V,\ol\omega_R}(v)(\delta_{W,\ol\omega_R}(w)\circ\nu)\\
&=\delta_{W,\ol\omega_R}(w)(\nu(v))
=\nu(v)(w)
\end{align*}
and hence $\nu=\delta_{W,\ol\omega_R}^{\ol\vee}\circ\nu^{\ol\vee\ol\vee}\circ\delta_{V,\ol\omega_R}$ as claimed.
\end{proof}


\begin{cor}\label{43}
Assume that $R$ is Gorenstein and consider an $I$-free approximation \eqref{5} of an $I$-reflexive $R$-module $M$.
Then $V$ in diagram~\eqref{24} is ($\ol\omega_R$-)reflexive if and only if $\nu$ in Lemma~\ref{36} identifies $V=W^{\ol\vee}$.
\end{cor}

\begin{proof}
The claim follows from Proposition~\ref{58} and Lemma~\ref{50}.
\end{proof}

\subsection{\texorpdfstring{$R$}{R}-dual \texorpdfstring{$I$}{I}-free approximation}\label{21}

In this subsection we consider the $R$-dual of an $I$-free approximation \eqref{5}.
The interesting part of the long exact Ext-sequence of $-^\vee$ applied to \eqref{5} turns out to be
\begin{equation}\label{15}
\xymatrix@C=1em{
	0 & \ar[l]\Ext^k_R(M,\omega_R) & \ar[l]\Ext^k_R(W,\omega_R) & \ar[l]_-\beta\Ext^{k-1}_R(IF,\omega_R) & \ar[l]\Ext^{k-1}_R(M,\omega_R) & \ar[l]0.
}
\end{equation}
In fact, applying $-^\vee$ to \eqref{4} yields (see Lemma~\ref{12} and \cite[Thm.~3.3.10.(c).(ii)]{BH93})
\[
\Ext_R^i(IF,\omega_R)=F^*\otimes_R\Ext_R^i(I,\omega_R)=F^*\otimes_R\Ext_R^{i+1}(\ol R,\omega_R)=0\text{ for } i\ne 0,k-1.
\]
In case both $R$ and $\ol R$ are Gorenstein, we will identify the map $\beta$ to its image with the map $\ol\alpha$ in \eqref{24} (see Corollary~\ref{49}).
In \S\ref{23} this fact will serve to relate the Gorenstein dimension of $M$ to the depth of $V$.

\smallskip

In order to describe the map $\beta$ in \eqref{15} we fix a canonical module $\omega_R$ of $R$ with an injective resolution $(E^\bullet,\partial^\bullet)$,
\[
\xymatrix{
0\ar[r] & \omega_R\ar[r] & E^0\ar[r]^{\partial^0} & E^1\ar[r]^{\partial^1} & E^2\ar[r]^{\partial^2} & \cdots
}.
\]
We use it to fix representatives 
\[
\Ext_R^i(-,\omega_R):=H^i\Hom_R(-,E^\bullet).
\]
Then (see \cite[Thms.~3.3.7.(b), 3.3.10.(c).(ii)]{BH93})
\begin{equation}\label{29}
H^i\Ann_{E^\bullet}(I)=H^i\Hom(\ol R,E^\bullet)=\Ext_R^i(\ol R,\omega_R)=\delta_{i,k}\cdot\omega_{\ol R}
\end{equation}
where
\[
\omega_{\ol R}:=H^k\Ann_{E^\bullet}(I)
\]
is a canonical module of $\ol R$.


In the sequel we explicit the maps of the following commutative diagram
\begin{equation}\label{20}
\xymatrix{
\Ext^k_R(W,\omega_R)\ar[ddd]_-{\gamma}^-{\cong} & & \Ext^{k-1}_R(IF,\omega_R)\ar[ll]_-{\beta}\\
&& F^*\otimes_R\Ext_R^{k-1}(I,\omega_R) \ar[u]_-{\chi}^-{\cong}\\
& & F^*\otimes_R H^{k-1}(E^\bullet/\Ann_{E^\bullet}(I))\ar[u]_-{F^*\otimes H^{k-1}(\tau^\bullet)}^-{\cong} \ar[d]^-{F^*\otimes\zeta}_-{\cong}\\
\llap{$\Hom_{\ol R}(W,\omega_{\ol R})=$}W^{\vee} & V'\ar@{_{(}->}[l]_-{\nu'} & F^*\otimes_R\omega_{\ol R}\ar@{->>}[l]_-{\alpha'}
}
\end{equation}
which defines the map $\nu'\circ\alpha'$ and its image $V'$.
The maps $\tau^\bullet$, $\chi$, $\zeta$, $\gamma$ and $\alpha'$ are described in Lemmas~\ref{10}, \ref{12}, \ref{31}, \ref{13} and Proposition~\ref{18} respectively.


\begin{lem}\label{10}
For any injective $R$-module $E$ there is a canonical isomorphism
\[
\tau\colon E/\Ann_E(I)\to\Hom_R(I,E),\quad\ol e\mapsto-\cdot e=(x\mapsto x\cdot e).
\]
In particular, there is a canonical isomorphism $\tau^\bullet\colon E^\bullet/\Ann_{E^\bullet}(I)\to\Hom_R(I,E^\bullet)$.
\end{lem}

\begin{proof}
Applying the exact functor $\Hom_R(-,E)$ to \eqref{4} yields a short exact sequence
\[
0\gets\Hom_R(I,E)\gets\Hom_R(R,E)\gets\Hom_R(\ol R,E)\gets0.
\]
Identifying $E=\Hom_R(R,E)$, $e\mapsto -\cdot e$, and hence 
\begin{equation}\label{42}
\Hom_R(\ol R,E)=\Ann_E(I)
\end{equation}
yields the claim.
\end{proof}


\begin{lem}\label{12}
For any $i\in\NN$ there is a canonical isomorphism
\begin{align*}
\chi_i\colon
F^*\otimes_R\Ext_R^i(I,\omega_R)=F^*\otimes_R H^i\Hom_R(I,E^\bullet)
&\to H^i\Hom_R(IF,E^\bullet)=\Ext_R^i(IF,\omega_R),\\
\varphi\otimes[\psi]&\mapsto[\varphi\vert_{IF}\cdot\wt\psi(1)]=[(\kappa\circ\iota)^*(\varphi)\cdot\wt\psi(1)]
\end{align*}
where $\wt\psi\in\Hom_R(R,E^\bullet)$ extends $\psi\in\Hom_R(I,E^\bullet)$.
We set $\chi:=\chi_{k-1}$.
\end{lem}

\begin{proof}
For any $i\in\NN$ there is a sequence of canonical isomorphisms
\begin{align*}
F^*\otimes_R H^i\Hom_R(I,E^\bullet)
&=\Hom_R(F,H^i\Hom_R(I,E^\bullet))\\
&= H^i\Hom_R(F,\Hom_R(I,E^\bullet))\\
&= H^i\Hom_R(IF,E^\bullet),
\end{align*}
the latter one being Hom-tensor adjunction, sending
\begin{align*}
\varphi\otimes[\psi] 
&\mapsto(f\mapsto\varphi(f)\cdot[\psi]=[\varphi(f)\cdot\psi])\\
&\mapsto[f\mapsto\varphi(f)\cdot\psi]\\
&\mapsto[x\cdot f\mapsto\varphi(f)\cdot\psi(x)=\varphi(x\cdot f)\cdot \wt\psi(1)]=[\varphi\vert_{IF}\cdot\wt\psi(1)]
\end{align*}
where $x\in I$ and $f\in F$.
\end{proof}


\begin{lem}\label{31}
There is a connecting isomorphism
\begin{align*}
\zeta\colon H^{k-1}(E^\bullet/\Ann_{E^\bullet}(I))&\to H^k\Ann_{E^\bullet}(I)=\omega_{\ol R},\\
[\ol e]&\mapsto[\partial^{k-1}(e)].
\end{align*}
\end{lem}

\begin{proof}
The connecting homomorphism $\zeta$ in degree $k$ of the short exact sequence
\[
0\to \Ann_{E^\bullet}(I)\to E^\bullet \to E^\bullet/\Ann_{E^\bullet}(I)\to 0
\]
is an isomorphism since $E^\bullet$ is a resolution and hence $H^i(E^\bullet)=0$ for $i\ge k-1\ge1$.
\end{proof}


\begin{lem}\label{13}
For any $\ol R$-module $N$ there is a canonical isomorphism
\begin{align*}
\gamma\colon H^k\Hom_R(N,E^\bullet)&\to\Hom_{\ol R}(N,H^k\Ann_{E^\bullet}(I))=N^\vee,\\
[\phi]&\mapsto(n\mapsto [\phi(n)]).
\end{align*}
\end{lem}

\begin{proof}
Fix an $\ol R$-projective resolution $(P_\star,\delta_\star)$ of $N$ and consider the double complex
\[
A^{\star,\bullet}:=
\Hom_R(P_\star,E^\bullet)=
\Hom_{\ol R}(P_\star,\Hom_R(\ol R,E^\bullet))=
\Hom_{\ol R}(P_\star,\Ann_{E^\bullet}(I))
\]
whose alternate representation is due to Hom-tensor adjunction and \eqref{42}.
It yields two spectral sequences with the same limit.
By exactness of $\Hom_{\ol R}(P_\star,-)$ and \eqref{29} and using the alternate representation the $E_2$-page of the first spectral sequence identifies with
\[
'E^{p,q}_2=
H^p(H^{\star,q}(A^{\star,\bullet}))=
H^p\Hom_{\ol R}(P_\star,H^q\Ann_{E^\bullet}(I))=
\delta_{k,q}\cdot H^p\Hom_{\ol R}(P_\star,\omega_{\ol R}).
\]
By exactness of $\Hom_R(-,E^\bullet)$ the $E_2$-page of the second spectral sequence reads
\[
''E^{p,q}_2=
H^q(H^{p,\bullet}(A^{\star,\bullet}))=
H^q\Hom_R(H^pP_\star,E^\bullet)=
\delta_{p,0}\cdot H^q\Hom_R(N,E^\bullet).
\]
So both spectral sequences degenerate.
The resulting isomorphism $''E_2^{0,k}\to{'E}_2^{0,k}$ is $\gamma$.
\end{proof}


\begin{prp}\label{18}
Assume that $R$ is Gorenstein and consider an $I$-free approximation \eqref{5}.
Then the map $\alpha'$ in diagram~\eqref{20} is induced by
\begin{align*}
\nu'\circ\alpha'\colon F^*\otimes_R\omega_{\ol R}=F^*\otimes_RH^k\Ann_{E^\bullet}(I)&\to\Hom_{\ol R}(W,H^k\Ann_{E^\bullet}(I))=W^{\vee},\\
\varphi\otimes[a]&\mapsto \ol\varphi\cdot [a],
\end{align*}
where $\varphi\mapsto\ol\varphi$ is \eqref{26} with $\omega_R=R$.
In particular, $\Ext^k_R(M,R)=0$ if $\nu'$ is surjective.
\end{prp}

\begin{proof}
The proof is done by chasing diagram~\eqref{20} and the diagram
\[
\xymatrix{
0\ar[r] & \Hom_R(W, E^{k-1})\ar[r]^-{\rho^*}\ar[d]^-{(\partial^{k-1})_*} & \Hom_R(M, E^{k-1})\ar[r]^-{\iota^*}\ar[d]^-{(\partial^{k-1})_*} & \Hom_R(IF, E^{k-1})\ar[r]\ar[d]^-{(\partial^{k-1})_*} & 0\\
0\ar[r] & \Hom_R(W, E^{k})\ar[r]^-{\rho^*} & \Hom_R(M, E^k)\ar[r]^-{\iota^*} & \Hom_R(IF, E^k)\ar[r] & 0.
} 
\]
This latter defines the connecting homomorphism $\beta$ in \eqref{15} on representatives as $(\rho^*)^{-1}\circ(\partial^{k-1})_*\circ(\iota^*)^{-1}$ where $(\iota^*)^{-1}$ denotes the choice of any preimage under $\iota^*$.

Let $\varphi\otimes [\ol e]\in F^*\otimes_R H^{k-1}(E^\bullet/\Ann_{E^\bullet}(I))$.
Then by Lemmas~\ref{10}, \ref{12}, \ref{31} and \ref{13}, and diagram~\eqref{17} with $\omega_R=R$
\[
\xymatrix@C=2em{
& [\kappa^*(\varphi)\cdot e]\ar@{|->}[r]^-{H^{k-1}(\iota^*)}\ar@{|->}[d] & [(\iota^*\circ\kappa^*)(\varphi)\cdot e]\\
[((\rho^{-1})^*\circ\kappa^*)(\varphi)\cdot\partial^{k-1}(e)]\ar@{|->}[dd]_-{\gamma}\ar@{|->}[r]^-{H^k(\rho^*)} & [\kappa^*(\varphi)\cdot\partial^{k-1}(e)] & \varphi\otimes[-\cdot e]\ar@{|->}[u]_-\chi\\
& & \varphi\otimes[\ol e]\ar@{|->}[u]_-{F^*\otimes H^{k-1}(\tau^\bullet)}\ar@{|->}[d]^-{F^*\otimes\zeta}\\
(\pi\circ\varphi\circ\kappa\circ\rho^{-1})\cdot[\partial^{k-1}(e)]=\ol\varphi\cdot[\partial^{k-1}(e)] & & \ar@{|->}[ll]_-{\nu'\circ\alpha'}\varphi\otimes[\partial^{k-1}(e)]
}
\]
where $\rho^{-1}$ denotes the choice of any preimage under $\rho$.
By diagram~\eqref{17} and Lemma~\ref{31} the ambiguity of this choice is cancelled when multiplying $(\rho^{-1})^*\circ\kappa^*(\varphi)=\varphi\circ\kappa\circ\rho^{-1}$ with $\partial^{k-1}(e)\in\Ann_{E^\bullet}(I)$.

The particular claim follows from diagram~\eqref{20} and the exact sequence~\eqref{15}.
\end{proof}


\begin{cor}\label{49}
Assume that both $R$ and $\ol R$ are Gorenstein and consider an $I$-free approximation \eqref{5}.
Then identifying $\ol\omega_R=\omega_{\ol R}$ (see diagrams~\eqref{24} and \eqref{20}) makes
\[
\alpha'=\ol\alpha,\quad V'=V,\quad\Ext_R^{k-1}(M,R)\cong\ker(\ol\alpha)=M^I/F^I.
\]
In particular, if $M$ is $I$-reflexive, then $\Ext^k_R(M,R)=0$ if and only if $V$ is ($\ol\omega_R$-)reflexive.
\end{cor}

\begin{proof}
Let $\varphi\mapsto\ol\varphi$ be \eqref{26} with $\omega_R=R$. 
Pick free generators $\varepsilon\in\omega_R$ and $\wt\varepsilon\in\omega_{\ol R}$ inducing the identification $\ol\omega_R=\omega_{\ol R}$ by sending $\ol\varepsilon=\pi_\omega(\varepsilon)\mapsto\wt\varepsilon$.
Then 
\begin{alignat*}{2}
F^\vee\otimes_R\ol R&=F^*\otimes_R\ol\omega_R=F^*\otimes_R\omega_{\ol R},\quad&
W^{\ol\vee}&=W^\vee,\\
(\varphi\cdot\varepsilon)\otimes\ol1&\leftrightarrow\varphi\otimes\ol\varepsilon\leftrightarrow\varphi\otimes\wt\varepsilon,\quad&
\ol\varphi\cdot\ol\varepsilon&\leftrightarrow\ol\varphi\cdot\wt\varepsilon.
\end{alignat*}
By diagram~\eqref{17} and Lemma~\ref{36} the map $F^\vee\otimes_R\ol R\to W^{\ol\vee}$ induced by $\nu\circ\alpha$ sends
\[
(\varphi\cdot\varepsilon)\otimes\ol1\mapsto\ol{\varphi\cdot\varepsilon}
=\pi_\omega\circ((\varphi\circ\kappa\circ\rho^{-1})\cdot\varepsilon)
=(\pi\circ\varphi\circ\kappa\circ\rho^{-1})\cdot\pi_\omega(\varepsilon)
=\ol\varphi\cdot\ol\varepsilon.
\]
By Proposition~\ref{18} this map coincides with $\nu'\circ\alpha'$ subject to the above identifications.
This shows that $\alpha'=\ol\alpha$ and $V'=V$.
By the exact sequence~\eqref{15}, the commutative diagram~\eqref{20} and the exact upper row of diagram~\eqref{24},
\begin{align*}
\Ext_R^{k-1}(M,R)&=\ker(\beta)\cong\ker(\alpha')=\ker(\ol\alpha)=M^I/F^I,\\
\Ext^k_R(M,R)&=\coker(\beta)\cong\coker(\nu')=W^\vee/\nu'(V').
\end{align*}
In particular $\Ext^k_R(M,R)=0$ if and only if $\nu'$ identifies $V'=W^\vee$ or, equivalently, if $\nu$ identifies $V=W^{\ol\vee}$.
The particular claim now follows with Corollary~\ref{43}.
\end{proof}

\subsection{Projective dimension and residual depth}\label{23}

Assume that $R$ is Gorenstein.
Then every finitely generated $R$-module $M$ has finite Gorenstein dimension $\Gdim(M)<\infty$ (see \cite[Thm.~17]{Mas00}).
Recall that if $M$ has finite projective dimension $\pdim(M)<\infty$, then $\Gdim(M)=\pdim(M)$ (see \cite[Cor.~21]{Mas00}).
Consider an $I$-free approximation~\eqref{5} of an $R$-module $M$.
In the following we relate the case of minimal Gorenstein dimension of $M$ to Cohen--Macaulayness of $V$, proving our main result.


\begin{lem}\label{45}
Assume that $R$ is Gorenstein and consider an $I$-free approximation~\eqref{5} with $W\ne0$.
Then $W$ is a maximal Cohen--Macaulay $\ol R$-module if and only if $\Gdim(M)\le k$.
In this case $\Gdim(M)\le k-1$ if and only if $\Ext^k_R(M,R)=0$.
If $\ol R$ is Gorenstein, then $\Gdim(M)\ge k-1$ unless $\ol\alpha$ in diagram~\eqref{24} is injective.
\end{lem}

\begin{proof}
By hypothesis $M\ne0$ is finitely generated over the Gorenstein ring $R$.
It follows that (see \cite[Thm.~17, Lem.~23.(c)]{Mas00}) 
\begin{equation}\label{62}
\Gdim(M)=\max\set{i\in\NN\xmid \Ext^i_R(M,R)\ne0}<\infty.
\end{equation}
The Auslander--Bridger Formula (see \cite[Thm.~29]{Mas00}) then states that 
\begin{equation}\label{30}
\depth(M)=\depth(R)-\Gdim(M)=\dim(R)-\Gdim(M)=n-\Gdim(M).
\end{equation}
By the Depth Lemma (see \cite[Prop.~1.2.9]{BH93}) applied to the short exact sequence~\eqref{4}
\begin{align*}
n-k+1=\depth(\ol R)+1
&\ge\min\set{\depth(R),\depth(I)-1}+1=\depth(I)\\
&\ge\min\set{\depth(R),\depth(\ol R)+1}=n-k+1
\end{align*}
and hence
\begin{equation}\label{38}
\depth(IF)=\depth(I)=n-k+1.
\end{equation}

\begin{asparaenum}

\item[($\implies$)] 
Using \eqref{38} and \eqref{30} the Depth Lemma applied to the short exact sequence \eqref{5} gives
\[
\Gdim(M)=n-\depth(M)\le n-\min\set{\depth(IF),\depth(W)}\le n-(n-k)=k.
\]

\item[($\impliedby$)] Using \eqref{30} and \eqref{38} the Depth Lemma applied to the short exact sequence~\eqref{5} gives 
\[
n-k=\dim(\ol R)\ge\dim(W)\ge\depth(W)\ge\min\set{\depth(M),\depth(IF)-1}\ge n-k.
\]

\end{asparaenum}

By \eqref{62} this latter inequality becomes $\Gdim(M)\le k-1$ if and only if $\Ext^k_R(M,R)=0$ (see \cite[Lem.~23.(c)]{Mas00}).

If $\ol R$ is Gorenstein and $\ol\alpha$ is not injective, then $\Ext^{k-1}_R(M,R)\ne0$ by Corollary~\ref{49} and hence $\Gdim(M)\ge k-1$ by \eqref{62}.
\end{proof}


We can now conclude the proof of our main result.

\begin{proof}[Proof of Theorem~\ref{60}]
Since $M$ is $I$-reflexive, $W=V^{\ol\vee}$ by Proposition \ref{58}.

\begin{asparaenum}

\item[($\implies$)] Suppose that $V$ is maximal Cohen--Macaulay.
Then also $W$ is maximal Cohen--Macaulay and $V$ is ($\ol\omega_R$-)reflexive (see \cite[Prop.~3.3.3.(b).(ii), Thm.~3.3.10.(d).(iii)]{BH93}).
By Corollary~\ref{49} $\Ext^k_R(M,R)=0$ and by Lemma~\ref{45} $\Gdim(M)=k-1$.

\item[($\impliedby$)] Suppose that $\Gdim(M)\le k-1$.
By Lemma~\ref{45} $W$ is maximal Cohen--Macaulay and $\Ext^k(M,R)=0$.
By Corollary~\ref{49} $V=W^{\ol\vee}$ is ($\ol\omega_R$-)reflexive and maximal Cohen--Macaulay (see \cite[Prop.~3.3.3.(b).(ii)]{BH93}).

\end{asparaenum}

The last claim is due to Lemma~\ref{45}.
\end{proof}

\subsection{Restricted \texorpdfstring{$I$}{I}-free approximation}\label{65}

In this subsection we describe a construction that reduces the support of an $I$-free approximation~\eqref{5} and preserves $I$-reflexivity of $M$ under suitable hypotheses.
In \S\ref{83} this will be related to the definition of multi-logarithmic differential forms and residues along Cohen--Macaulay spaces (see \cite[\S10]{Ale14} and \cite[Ch.~4]{Pol16}).

Fix an ideal $J\unlhd R$ with $I\subseteq J$ and set $S:=\ol R$ and $T:=R/J$.
By hypothesis $S$ is Cohen--Macaulay and hence (see \cite[Prop.1.2.13]{BH93})
\begin{equation}\label{82}
\Ass(S)=\Min\Spec(S).
\end{equation}


\begin{lem}\label{46}
There is an inclusion
\[
\Supp_S(T)\cap\Ass(S)\subseteq\Ass_S(T).
\]
In particular, equality in $\Hom_S(N,S)$ for any $T$-module $N$, or in $\Hom_S(N,T)$ for any $S$-module $N$, can be checked at $\Ass_S(T)$.
\end{lem}

\begin{proof}
The inclusion follows from \eqref{82} and $\Min\Supp_S(T)\subseteq\Ass_S(T)$.
For any $T$-module $N$ (see \cite[Exe.~1.2.27]{BH93})
\[
\Ass_S(\Hom_S(N,S))=\Supp_S(N)\cap\Ass(S)\subseteq\Supp_S(T)\cap\Ass(S)\subseteq\Ass_S(T)
\]
and the first particular claim follows, the second holds for a similar reason.
\end{proof}


\begin{dfn}\label{39}
For any $S$-module $N$ we consider the submodule supported on $V(J)$
\[
N_T:=\Hom_S(T,N)=\Ann_N(J)\subseteq N.
\]
For an $I$-free approximation~\eqref{5} its \emph{$J$-restriction} is the $I$-free approximation
\begin{equation}\label{108}
\xymatrix{
0\ar[r] & IF\ar[r]^-{\iota_J} & M_J\ar[r]^-{\rho_T} & W_T\ar[r] & 0
}
\end{equation}
defined as its image under the map $\Ext^1_R(W,IF)\to\Ext^1_R(W_T,IF)$.
\end{dfn}


In explicit terms it is the source of a morphism of $I$-free approximations
\begin{equation}\label{40}
\xymatrix{
0\ar[r] & IF\ar[r]^-\iota & M\ar[r]^-\rho & W\ar[r] & 0 \\
0\ar[r] & IF\ar@{=}[u]\ar[r]^-{\iota_J} & M_J\ar@{^{(}->}[u]\ar[r]^-{\rho_T} & W_T\ar@{^{(}->}[u]\ar[r] & 0. 
}	
\end{equation}
The right square is obtained as the pull-back of $\rho$ and $W_T\into W$, whose universal property applied to $\iota$ and $0\colon IF\to W_T$ gives the left square.
The analogue of $\kappa$ in \eqref{6} for the $J$-restriction~\eqref{108} is the composition
\begin{equation}\label{81}
\xymatrix{
\kappa_J\colon M_J=IF:_MJ\subseteq M\ar[r]^-\kappa & F.
}
\end{equation}


By Lemma \ref{7} and the Snake Lemma, applying $-^I$ to \eqref{40} yields (see Definition~\ref{9})
\begin{equation}\label{41}
\xymatrix{
0 & V\ar[l]\ar@{->>}[d] & F^\vee\ar[l]_-{\alpha}\ar@{=}[d] & M^I\ar[l]_-{\lambda}\ar@{^{(}->}[d] & 0\ar[l]\\
0 & V^T\ar[l] & F^\vee\ar[l]_{\alpha^T} & M_J^I\ar[l]_{\lambda^J} & 0\ar[l]
}
\end{equation}
where the bottom row 
\begin{equation}\label{109}
\xymatrix{
0 & V^T\ar[l] & F^\vee\ar[l]_{\alpha^T} & M_J^I\ar[l]_{\lambda^J} & 0\ar[l]
}
\end{equation}
is the $I$-dual~\eqref{37} of the $J$-restriction~\eqref{108}.
In diagram~\eqref{41}, we denote
\begin{equation}\label{89}
U:=\ker(V\onto V^T).
\end{equation}


The $J$-restriction behaves well under the following hypothesis on $T$.
\begin{equation}\label{67}
T_\pp=
\begin{cases}
S_\pp & \text{ if } \pp\in\Ass_S(T),\\
0 & \text{ if } \pp\in\Ass(S)\setminus\Ass_S(T).
\end{cases}
\end{equation}
This is due to the following 

\begin{rmk}\label{48}
Our constructions commute with localization.
As special cases of the $J$-restriction and its $I$-dual we record
\[
(\iota_J,\rho_T)=
\begin{cases}
(\iota,\rho) & \text{ if } T=S,\\
(\id_{IF},0) & \text{ if } T=0,
\end{cases}
\quad
(\lambda^J,\alpha^T)=
\begin{cases}
(\lambda,\alpha) & \text{ if } T=S,\\
(\id_{F^\vee},0) & \text{ if } T=0.
\end{cases}
\]
Localizing \eqref{40} and \eqref{41} at the image of $\pp\in\Ass(S)$ under the map $\Spec(S)\to\Spec(R)$ yields these special cases under hypothesis~\eqref{67}.
\end{rmk}


In the setup of our applications in \S\ref{51} condition~\eqref{67} holds true due to the following

\begin{lem}\label{91}
If $S$ is reduced and $T$ is unmixed with $\dim(T)=\dim(S)$, then condition~\eqref{67} holds and $\Ass_S(T)\subseteq\Ass(S)$.
\end{lem}

\begin{proof}
By hypothesis on $T$ and \eqref{82} 
\begin{equation}\label{85}
\Ass_S(T)=\Min\Supp_S(T)\subseteq\Min\Spec(S)=\Ass(S).
\end{equation}
By hypothesis on $S$, for any $\pp\in\Ass(S)$, $S_\pp$ is a field with factor ring $T_\pp$.
If $\pp\in\Ass_S(T)$, then $T_\pp\ne0$ and hence $T_\pp=S_\pp$.
Otherwise, $\pp\not\in\Supp_S(T)$ by \eqref{85} and hence $T_\pp=0$.
\end{proof}


\begin{lem}\label{84}
Assume that $R$ is Gorenstein and consider the $J$-restriction~\eqref{108} of an $I$-free approximation. 
If $T$ satisfies condition~\eqref{67}, then for $U$ as defined in \eqref{89}
\[
\alpha^{-1}(U)=\set{\varphi\in F^\vee\xmid\varphi\circ\kappa(M)\subseteq J\omega_R}.
\]
In particular, $JV\subseteq U$.
\end{lem}

\begin{proof}
Let $\varphi\in F^\vee$ and denote by $\ol\varphi_T$ the map $\ol\varphi$ in diagram~\eqref{17} for the $J$-restriction~\eqref{108}.
Consider the map $\psi$ defined by the commutative diagram
\begin{equation}\label{118}
\xymatrix{
W\ar[r]^-\psi\ar[dr]^-{\ol\varphi} & T\otimes_R\omega_R\\
W_T\ar[r]^-{\ol\varphi_T}\ar@{^(->}[u] & S\otimes_R\omega_R.\ar@{->>}[u]
}
\end{equation}
By Lemma~\ref{46} and since $\omega_R\cong R$ both $\ol\varphi_T=0$ and $\psi=0$ can be checked at $\Ass_S(T)$.
There the vertical maps in diagram~\eqref{118} induce the identity by condition~\eqref{67} and Remark~\ref{48}.
With diagram~\eqref{41}, Lemma~\ref{36} applied to \eqref{108} and diagram~\eqref{17} it follows that
\[
\alpha(\varphi)\in U
\iff\alpha^T(\varphi)=0
\iff\ol\varphi_T=0
\iff\psi=0
\iff\varphi\circ\kappa(M)\subseteq J\omega_R.
\]
This proves the equality and the inclusion follows with $JV=J\alpha(F^\vee)=\alpha(JF^\vee)$.
\end{proof}


\begin{prp}\label{72}
Assume that $R$ is Gorenstein and consider the $J$-restriction~\eqref{108} of an $I$-free approximation.
If $T$ satisfies condition~\eqref{67}, then with $M$ also $M_J$ is $I$-reflexive.
\end{prp}


\begin{proof}
By Lemma~\ref{84} there is a short exact sequence
\begin{equation}\label{119}
0\to U/JV\to V/JV\to V^T\to 0.
\end{equation}
By condition~\eqref{67} and Remark~\ref{48}
\[
JS_\pp =
\begin{cases}
0 & \text{ if }\pp\in\Ass_S(T),\\
S_\pp & \text{ if }\pp\in\Ass(S)\setminus\Ass_S(T),
\end{cases}
\quad
(V\onto V^T)_\pp=
\begin{cases}
\id_{V_\pp} & \text{ if }\pp\in\Ass_S(T),\\
0 & \text{ if }\pp\in\Ass(S)\setminus\Ass_S(T),
\end{cases}
\]
and hence 
\begin{align*}
\forall\pp\in\Ass(S)\colon(JV)_\pp=JS_\pp V_\pp=U_\pp
&\implies(U/JV)_\pp=0\\
&\implies\dim(U/JV)<\dim(S)=\depth(\ol\omega_R).
\end{align*}
Then $(U/JV)^{\ol\vee}=0$ by Ischebeck's Lemma (see \cite[Satz 1.9]{HK71}).
Using sequence~\eqref{119} and Hom-tensor adjunction it follows that
\[
(V^T)^{\ol\vee}=(V/JV)^{\ol\vee}=(T\otimes_SV)^{\ol\vee}=(V^{\ol\vee})_T.
\]
Denote by $\nu_T$ the map $\nu$ from Lemma~\ref{36} applied to the $J$-restriction~\eqref{108}.
We obtain a diagram
\begin{equation}\label{78}
\xymatrix{
W_T\ar[rr]^-{(\nu^{\ol\vee}\circ\delta_{W,\ol\omega_R})_T}\ar@{=}[d] && (V^{\ol\vee})_T \ar@{=}[d]\\
W_T\ar[r]^-{\delta_{W_T,\ol\omega_R}} & (W_T)^{\ol\vee\ol\vee}\ar[r]^-{(\nu_T)^{\ol\vee}} & (V^T)^{\ol\vee}.
}
\end{equation}
By Lemma~\ref{46} and since $\ol\omega_R\cong S$, its commutativity can be checked at $\Ass_S(T)$.
By condition~\eqref{67} and Remark~\ref{48} top and bottom horizontal maps in diagram~\eqref{78} identify at $\Ass_S(T)$.
Diagram~\eqref{78} thus commutes and Proposition~\ref{58} yields the claim.
\end{proof}


The Cohen--Macaulay property is invariant under restriction of scalars $S\to T$ and by Hom-tensor adjunction $\Hom_S(-,\omega_S)=\Hom_T(-,\omega_T)$ on $T$-modules where (see \cite[Thm.~3.3.7.(b)]{BH93})
\begin{equation}\label{116}
\omega_T=\Hom_S(T,\omega_S).
\end{equation}
Combining Theorem~\ref{60} and Proposition~\ref{72} yields the following (see diagram~\eqref{41})

\begin{cor}\label{61}
In addition to the hypotheses of Theorem~\ref{60}, let $J\unlhd R$ with $J\subseteq I$ be such that $T=R/J$ satisfies condition~\eqref{67} and $W_T\ne0$.
Consider the $J$-restriction~\eqref{108} with $I$-dual~\eqref{109}.
Then $W_T=\Hom_T(V^T,\omega_T)$ and $V^T$ is a maximal Cohen--Macaulay $T$-module if and only if $\Gdim(M_J)\le k-1$.
In this latter case $V^T=\Hom_T(W_T,\omega_T)$ is $\omega_T$-reflexive.
Unless $T\otimes\alpha^T$ (and hence $\ol\alpha$) is injective $\Gdim(M_J)\ge k-1$.\qed
\end{cor}


Finally we mention a construction analogous to Definition~\ref{39} not used in the sequel.

\begin{rmk}
Assume that $J$ satisfies the hypotheses on $I$ and consider an $I$-free approximation~\eqref{5} where $W$ is already a $T$-module.
Then $W_T=W$ and $M_J=M$ and the image of \eqref{5} under the map $\Ext^1_R(W,IF)\to\Ext^1_R(W,JF)$ is a $J$-free approximation that fits into a commutative diagram with cartesian left square
\[
\xymatrix{
0\ar[r] & JF\ar[r] & M^J\ar[r] & W\ar[r] & 0\\
0\ar[r] & IF\ar[r]\ar@{^{(}->}[u] & M\ar[r]\ar@{^{(}->}[u] & W\ar[r]\ar@{=}[u] & 0
}
\]
where $M^J/M_J\cong JF/IF$.
In particular, $M^J=M_J$ if and only if $I=J$.
\end{rmk}

\section{Application to logarithmic forms}\label{51}

In this section results from \S\ref{11} are used to give a more conceptual approach to and to generalize a duality of multi-logarithmic forms found by Pol~\cite{Pol16} as a generalization of result by Granger and the first author~\cite{GS14}.


Let $Y$ be a germ of a smooth complex analytic space of dimension $n$. 
Then $Y\cong(\CC^n,0)$ and $\OO_Y\cong\CC\set{x_1,\dots,x_n}$ by a choice of coordinates $x_1,\dots,x_n$ on $Y$. 
We denote by 
\[
\QQ_-:=Q(\OO_-)
\]
the total ring of fractions of $\OO_-$.
In this section we set $-^*:=\Hom_{\OO_Y}(-,\OO_Y)$.


Let $\Omega^\bullet_Y$ denote the \emph{De~Rham algebra on $Y$}, that is,
\[
\OO_Y\to\Omega^1_Y,\quad f\mapsto df,
\]
is the universally finite $\CC$-linear derivation of $\OO_Y$ (see \cite[\S2]{SS72} and \cite[\S11]{Kun86}) and $\Omega^q_Y=\bigwedge^q_{\OO_Y}\Omega^1_Y$ for all $q\ge 0$.
In terms of coordinates $\Omega^1_Y\cong\bigoplus_{i=1}^n\OO_Ydx_i$ and hence 
\[
\Omega^q_Y=\bigwedge^q_{\OO_Y}\Omega^1_Y\cong\bigoplus_{i_1<\dots<i_q}\OO_Y dx_{i_1}\wedge\dots\wedge dx_{i_q}
\]
is a free $\OO_Y$-module.
By definition the dual
\[
(\Omega^1_Y)^*=\Der_\CC(\OO_Y)=:\Theta_Y\cong\bigoplus_{i=1}^n\OO_Y\frac{\partial}{\partial x_{i}}
\]
is the module of \emph{$\CC$-linear derivations on $\OO_Y$}, or of vector fields on $Y$.
The module of \emph{$q$-vector fields on $Y$} is then the free $\OO_Y$-module
\[
(\Omega^q_Y)^*=\bigwedge^q_{\OO_Y}\Theta_Y=:\Theta^q_Y\cong\bigoplus_{i_1<\dots<i_q}\OO_Y \frac{\partial}{\partial x_{i_1}}\wedge\dots\wedge \frac{\partial}{\partial x_{i_q}}.
\]


\begin{ntn}\label{117}
We set $N:=\set{1,\dots,n}$ and $N^q_<:=\set{\ul j\in N^q\xmid j_1<\dots<j_q}$.
For $\ul j\in N^q$ and $\ul f=(f_1,\dots,f_\ell)\in\OO_Y^\ell$ we abbreviate
\begin{alignat*}{2}
dx_{\ul j}&:=dx_{j_1}\wedge\dots\wedge dx_{j_q},\quad&\frac{\partial}{\partial x_{\ul j}}&:=\frac{\partial}{\partial x_{j_1}}\wedge\dots\wedge\frac{\partial}{\partial x_{j_q}},\\
\ul j_{\hat{i}}&:=(j_1,\dots,\wh{j_i},\dots,j_q),&d\ul f&=df_1\wedge\dots\wedge df_\ell.
\end{alignat*}
\end{ntn}

The perfect pairing
\begin{equation}\label{115}
\Theta_Y^q\times\Omega^q_Y\to\OO_Y,\quad(\delta,\omega)\mapsto\ideal{\delta,\omega},
\end{equation}
then satisfies
\begin{equation}\label{120}
\ideal{\frac{\partial}{\partial x_{\ul j}},dx_{\ul k}}=\delta_{\ul j,\ul k}:=\delta_{j_1,k_1}\cdots\delta_{j_q,k_q}.
\end{equation}

\subsection{Log forms along complete intersections}\label{79}

Let $C\subseteq Y$ be a reduced complete intersection of codimension $k\ge1$.
Then $\OO_C=\OO_Y/\I_C$ where $\I_C=\I_{C/Y}$ is the ideal of $C\subseteq Y$.
Let $\ul h=(h_1,\dots,h_k)\in\OO_Y^k$ be any regular sequence such that $\I_C=\ideal{h_1,\dots,h_k}$.
Geometrically  $C=D_1\cap\dots\cap D_k$ where $D_i:=\set{h_i=0}$ for $i=1,\dots,k$.


\begin{ntn}
We denote $D:=D_1\cup\dots\cup D_k=\set{h=0}$ where $h:=h_1\cdots h_k$,
\begin{alignat*}{2}
-(D)&:=-\otimes_{\OO_Y}\OO_Y\frac1h, & -(-D)&:=-\otimes_{\OO_Y}\OO_Yh,\\
\Sigma=\Sigma_{C/D/Y}&:=\I_C(D)=\sum_{i=1}^k\frac{h_i}{h}\OO_Y\subseteq\QQ_Y, & -^\Sigma&:=\Hom_{\OO_Y}(-,\Sigma).
\end{alignat*}
Note that $\Sigma=\OO_Y$ in case $k=1$.
\end{ntn}


The following definition due to Aleksandrov (see \cite[\S3]{Ale12} and \cite[Def.~3.1.4]{Pol16}) generalizes Saito's logarithmic differential forms (see \cite{Sai80}) from the hypersurface to the complete intersection case.

\begin{dfn}\label{52}
The module of \emph{multi-logarithmic differential $q$-forms on $Y$ along $C$} is defined by
\begin{align*}
\Omega^q(\log C)=\Omega_Y^q(\log C)&:=\set{\omega\in\Omega^q_Y\xmid d\I_C\wedge\omega\subseteq\I_C\Omega^{q+1}_Y}(D)\\
&=\set{\omega\in\Omega^q_Y(D)\xmid\forall i=1,\dots,k\colon dh_i\wedge\omega\in\Sigma\Omega^{q+1}_Y}
\end{align*}
where the equality is due to the Leibniz rule.
Observe that
\[
\Sigma\Omega^q_Y\subseteq \Omega^q(\log C)\subseteq\QQ_Y\otimes_{\OO_Y}\Omega_Y^q
\]
with $\Omega^q(\log C)(-D)\subseteq\QQ_Y\otimes_{\OO_Y}\Omega_Y^q$ independent of $D$ (see \cite[Prop.~3.1.10]{Pol16}).
\end{dfn}

Extending Saito's theory (see \cite[\S1-2]{Sai80}) Aleksandrov (see \cite[\S 3-4,6]{Ale12}) gives an explicit description of multi-logarithmic differential forms and defines a multi-logarithmic residue map.
We summarize his results.

\begin{prp}\label{53}
An element $\omega\in\Omega^q_Y(D)$ lies in $\Omega^q(\log C)$ if and only if there exist $g\in \OO_Y$ inducing a non zero-divisor in $\OO_C$, $\xi\in \Omega^{q-k}_Y$ and $\eta\in\Sigma\Omega^q_Y$ such that
\[
g\omega=\frac{d\ul h}{h}\wedge \xi+\eta.
\]
This representation defines a \emph{multi-logarithmic residue map} 
\[
\res_C^q\colon\Omega^q(\log C)\to\QQ_C\otimes_{\OO_C}\Omega_C^{q-k},\quad \omega\mapsto\frac\xi g,
\]
that fits into a short exact \emph{multi-logarithmic residue sequence}
\begin{equation}\label{123}
\xymatrix
{
0\ar[r] & \Sigma\Omega_Y^q\ar[r] & \Omega^q(\log C)\ar[r]^-{\res_C^q} & \omega_C^{q-k}\ar[r] & 0
}
\end{equation}
where $\omega_C^p$ is the module of regular meromorphic $p$-forms on $C$.\qed
\end{prp}


\begin{cor}\label{125}
For $q<k$, $\Omega^q(\log C)=\Sigma\Omega_Y^q$ and $\Omega^n(\log C)=\Omega_Y^n(D)$.\qed
\end{cor}


\begin{rmk}\label{59}
The multi-logarithmic residue map can be written in terms of residue symbols as $\res_C^q(\omega)=\begin{bmatrix}h\omega\\\ul h\end{bmatrix}$ (see \cite[\S1.2]{Sch16}\footnote{This remark was made in the first author's talk \enquote{Normal crossings in codimension one} at the 2012 Oberwolfach conference \enquote{Singularities} (see \cite{Sch12}).}).
In particular $\res_C^k(\frac{dh}h)=\begin{bmatrix}dh\\\ul h\end{bmatrix}\in\omega_C^k$ is the fundamental form of $C$ (see \cite[\S5]{Ker83b}).\qed
\end{rmk}


Higher logarithmic derivation modules play a prominent role in arrangement theory (see for instance \cite{ATW07}).
Here we extend the definitions of Granger and the first author (see \cite[\S 5]{GS12}) and by Pol (see \cite[Def.~3.2.1]{Pol16}) as follows.

\begin{dfn}\label{54}
We define the module of \emph{multi-logarithmic $q$-vector fields on $Y$ along $C$} by
\begin{align*}
\Der^q(-\log C)=\Der_Y^q(-\log C)&:=\set{\delta\in\Theta^q_Y\xmid\ideal{\delta,\wedge^kd\I_C\wedge\Omega^{q-k}_Y}\subseteq\I_C}\\
&=\set{\delta\in\Theta^q_Y\xmid\ideal{\delta,d\ul h\wedge \Omega_Y^{q-k}}\subseteq\I_C}
\end{align*}
where the equality is due to the Leibniz rule.
Observe that
\[
\I_C\Theta_Y^q\subseteq\Der^q(-\log C).
\]
\end{dfn}


\begin{lem}\label{86}
We can identify the functors on $\OO_Y$-modules (see Notation~\ref{99})
\[
-^\Sigma=-(-D)^{\I_C}, \quad (\Sigma\otimes_{\OO_Y}-)^\Sigma=-^*,
\]
and hence $-^{\Sigma\Sigma}=-^{\I_C\I_C}$.
\end{lem}

\begin{proof}
Since $\OO_Y(D)$ is invertible and by Hom-tensor adjunction
\[
-^\Sigma=\Hom_{\OO_Y}(-,\I_C(D))
=\Hom_{\OO_Y}(-,\Hom_{\OO_Y}(\OO_Y(-D),\I_C))
=-(-D)^{\I_C}
\]
By Lemma~\ref{2} in case $k\ge2$, $\OO_Y=\I_C^{\I_C}=\Sigma^\Sigma$ and again by Hom-tensor adjunction
\[
(\Sigma\otimes_{\OO_Y}-)^\Sigma=\Hom_{\OO_Y}(\Sigma\otimes_{\OO_Y}-,\Sigma)=\Hom_{\OO_Y}(-,\Sigma^\Sigma)=-^*.\qedhere
\]
\end{proof}


\begin{lem}\label{55}
Any elements $\delta\in \Der^q(-\log C)$ and $\omega\in \Omega^q(\log C)$ pair to $\ideal{\delta,\omega}\in\Sigma$.
\end{lem}

\begin{proof}
Let $g$, $\xi$ and $\eta$ be as in Proposition~\ref{53}.
Then by definition
\[
g\ideal{\delta,h\omega}=
\ideal{\delta,h g\omega}=
\ideal{\delta,d\ul h\wedge\xi+h\eta}=
\ideal{\delta,d\ul h\wedge\xi}+h\ideal{\delta,\eta}
\in\I_C.
\]
Since $g$ induces a non zero-divisor in $\OO_C=\OO_Y/\I_C$ this implies that  $\ideal{\delta,h\omega}\in\I_C$ and hence $\ideal{\delta,\omega}\in\frac1h\I_C=\Sigma$.
\end{proof}


The following proofs for $q\ge k\ge1$ proceed along the lines of Saito's base case $q=k=1$ (see \cite[(1.6)]{Sai80}) or Pol's generalization to $q=k\ge1$ (see \cite[Prop.~3.2.13]{Pol16}).

\begin{lem}\label{98}
If $\omega\in\Omega^q_Y(D)$ with $\ideal{\Der^q(-\log C),\omega}\subseteq\Sigma$, then $\omega\in\Omega^q(\log C)$.
\end{lem}

\begin{proof}
For every $\ell\in\set{1,\dots,k}$ and $\ul j\in N^{q+1}_<$ consider
\[
\delta_{\ul j}^\ell:=\sum_{i=1}^{q+1}(-1)^{i+1}\frac{\partial h_\ell}{\partial x_{j_i}}\frac{\partial}{\partial x_{\ul j_{\hat{i}}}}\in\Theta_Y^q.
\]
For every $\ul i\in N^{q-k}$
\[
d\ul h\wedge dx_{\ul i}=\sum_{\ul k\in N^q_<}\frac{\partial(\ul h,x_{\ul i})}{\partial x_{\ul k}}dx_{\ul k},
\]
where $\frac{\partial(\ul h,x_{\ul i})}{\partial x_{\ul k}}$ is the $q\times q$-minor of the Jacobian matrix of $(\ul h,x_{\ul i})$ with column indices $\ul k$, and hence using \eqref{120}
\begin{align*}
\ideal{\delta^\ell_{\ul j},d\ul h\wedge dx_{\ul i}}
&=\sum_{i=1}^{q+1}(-1)^{i+1}\frac{\partial h_\ell}{\partial x_{j_i}}\sum_{\ul k\in N^q_<}\frac{\partial(\ul h,x_{\ul i})}{\partial x_{\ul k}}\ideal{\frac{\partial}{\partial x_{\ul j_{\hat{i}}}},dx_{\ul k}}
\\
&=\sum_{i=1}^{q+1}(-1)^{i+1}\frac{\partial h_\ell}{\partial x_{j_i}}\frac{\partial(\ul h,x_{\ul i})}{\partial x_{\ul j_{\hat{i}}}}
=\frac{\partial(h_\ell,\ul h,x_{\ul i})}{ \partial x_{\ul j}}=0.
\end{align*}
It follows that $\delta^\ell_{\ul j}\in\Der^q(-\log C)$ for all $\ell=1,\dots,k$ and $\ul j\in N^{q+1}_<$.

Now let $\omega=\sum_{\ul k\in N^q_<}\frac{a_{\ul k}}{h}dx_{\ul k}\in\Omega^q_Y(D)$ where $a_{\ul k}\in\OO_Y$.
For all $\ell=1,\dots,k$ and $\ul j\in N^{q+1}_<$
\[
\ideal{\delta^\ell_{\ol j},\omega}=
\sum_{i=1}^{q+1}(-1)^{i+1}\frac{\partial h_\ell}{\partial x_{j_i}}\sum_{\ul k \in N^q_<}\frac{a_{\ul k}}{h}\ideal{\frac{\partial}{\partial x_{\ul j_{\hat{i}}}},dx_{\ul k}}=
\sum_{i=1}^{q+1}(-1)^{i+1}\frac{\partial h_\ell}{\partial x_{j_i}}
\frac{a_{\ul j_{\hat{i}}}}h
\]
by \eqref{120} and hence
\begin{align*}
dh_\ell\wedge\omega&=\sum_{j=1}^n\frac{\partial h_\ell}{\partial x_j}dx_j\wedge\sum_{\ul k\in N^q_<}\frac{a_{\ul k}}{h}dx_{\ul k}
=\sum_{\ul j\in N^{q+1}_<}\sum_{i=1}^{q+1}\frac{\partial h_\ell}{\partial x_{j_i}}\frac{a_{\ul j_{\hat{i}}}}hdx_{j_i}\wedge dx_{\ul j_{\hat{i}}}\\
&=\sum_{\ul j\in N^{q+1}_<}\sum_{i=1}^{q+1}(-1)^{i+1}\frac{\partial h_\ell}{\partial x_{j_i}}\frac{a_{\ul j_{\hat{i}}}}hdx_{\ul j}
=\sum_{\ul j\in N^{q+1}_<}\ideal{\delta^\ell_{\ul j},\omega}dx_{\ul j}.
\end{align*}
If $\ideal{\Der^q(-\log C),\omega}\subseteq\Sigma$, then $dh_\ell\wedge\omega\in\Sigma\Omega_Y^q$ for all $\ell=1,\dots,k$ and hence $\omega\in\Omega^q(\log C)$.
\end{proof}


\begin{prp}\label{77}
There are chains of $\OO_Y$-submodules of $\QQ_Y\otimes_{\OO_Y}\Omega_Y^q$ and $\QQ_Y\otimes_{\OO_Y}\Theta_Y^q$ 	
\begin{gather}
\label{56}
\Omega^q_Y\subseteq\Sigma\Omega^q_Y\subseteq\Omega^q(\log C)\subseteq\Omega^q_Y(D)\subseteq \Sigma\Omega^q_Y(D),\\
\label{57}
\Sigma\Theta^q_Y\supseteq\Theta^q_Y\supseteq\Der^q(-\log C)\supseteq\I_C\Theta^q_Y\supseteq\Theta^q_Y(-D)
\end{gather}
that are $\Sigma$-duals of each other.
\end{prp}

\begin{proof}
Tensoring with $\QQ_Y$ makes both chains collapse.
The cokernels of all inclusions are therefore torsion whereas $\Sigma$ is torsion free.
Applying $-^\Sigma$ thus results in a chain of $\OO_Y$-modules again.
In case of \eqref{56} this yields
\[
(\Omega^q_Y)^\Sigma\supseteq (\Sigma\Omega^q_Y)^\Sigma\supseteq \Omega^q_Y(\log C)^\Sigma \supseteq\Omega^q_Y(D)^\Sigma\supseteq (\Sigma\Omega^q_Y(D))^\Sigma
\]
and, with Lemma \ref{86} and freeness of $\Omega^q_Y$ and $\Theta_Y^q$, the chain of $\OO_Y$-submodules of $\QQ_Y\otimes_{\OO_Y}\Theta_Y^q$
\[
\Sigma\Theta^q_Y\supseteq \Theta^q_Y\supseteq\Omega^q_Y(\log C)^\Sigma
\supseteq \I_C\Theta^q_Y\supseteq \Theta^q_Y(-D).
\]
For every $\delta\in\Omega^q(\log C)^\Sigma$ and $\xi\in \Omega^{q-k}$, $\frac{d\ul h}h\wedge\xi\in\Omega^q(\log C)$ by Proposition~\ref{53}, hence  
\[
\ideal{\delta,d\ul h\wedge\xi}=h\ideal{\delta,\frac{d\ul h}h\wedge\xi}\in h\Sigma=\I_C
\]
and $\delta\in\Der^q(-\log C)$.
With Lemma~\ref{55}, it follows that $\Omega^q_Y(\log C)^\Sigma=\Der^q(-\log C)$.

By the same reasoning $-^\Sigma$ applied to \eqref{57} yields a chain of $\OO_Y$-modules
\[
(\Sigma\Theta^q_Y)^\Sigma\subseteq(\Theta^q_Y)^\Sigma\subseteq\Der^q(-\log C)^\Sigma\subseteq(\Sigma\Theta^q_Y)(-D)^\Sigma\subseteq\Theta^q_Y(-D)^\Sigma
\]
that can be rewritten as the chain of $\OO_Y$-submodules of $\QQ_Y\otimes_{\OO_Y}\Omega_Y^q$
\[
\Omega^q_Y\subseteq\Sigma\Omega^q_Y\subseteq\Der^q(-\log C)^\Sigma\subseteq \Omega^q_Y(D)\subseteq\Sigma\Omega^q_Y(D).
\]
The missing equality $\Der^q(-\log C)^\Sigma=\Omega^q(\log C)$ follows from Lemmas~\ref{55} and \ref{98}.
\end{proof}

\subsection{Log forms along Cohen--Macaulay spaces}\label{83}

Let $X\subseteq Y$ be a reduced Cohen-Macaulay germ of codimension $k\ge2$.
Then $\OO_X=\OO_Y/\I_X$ where $\I_X:=\I_{X/Y}$ denotes the ideal $X\subseteq Y$.
There is a reduced complete intersection $C\subseteq Y$ of codimension $k$ such that $X\subseteq C$ and hence $\I_X\supseteq\I_C$ (see \cite[Prop. 4.2.1]{Pol16}).
Set $X':=\ol{C\setminus X}$ such that $C=X\cup X'$.
The link with \S\ref{65} is made by setting
\[
S:=\OO_C,\quad T:=\OO_X.
\]
By Lemma~\ref{91} condition \eqref{67} holds and
\begin{equation}\label{68}
\QQ_C=\prod_{\pp\in\Ass_{\OO_C}(\OO_X)}\OO_{X,\pp}\times\prod_{\pp\in\Ass_{\OO_C}(\OO_{X'})}\OO_{X',\pp}=\QQ_X\times\QQ_{X'}.
\end{equation}
This decomposition extends to differential forms as follows.


\begin{lem}\label{93}
We have $\QQ_Xd\I_C=\QQ_Xd\I_X\subseteq\QQ_X\otimes_{\OO_Y}\Omega_Y^1$
and hence
\[
\QQ_C\otimes_{\OO_C}\Omega_C^p=\QQ_X\otimes_{\OO_X}\Omega_X^p\oplus\QQ_{X'}\otimes_{\OO_{X'}}\Omega_{X'}^p.
\]
\end{lem}

\begin{proof}
By \eqref{68} we may localize at $\pp\in\Ass_{\OO_C}(\OO_X)$.
We may further assume $p=1$ since exterior product commutes with extension of scalars.
Let $\pp\mapsto\qq$ under $\Spec(\OO_C)\to\Spec(\OO_Y)$.
Then $\I_{C,\qq}=\I_{X,\qq}$ by \eqref{68} and hence $u\I_X\subseteq\I_C$ for some $u\in\OO_Y\setminus\qq$.
By the Leibniz rule $u d\I_X\subseteq d\I_C+\I_Xdu$ and hence the first claim. 
Since $\Omega_C^1=\Omega_Y^1/(\OO_Yd\I_C+\I_C\Omega_Y^1)$ this yields $\Omega_{C,\pp}^1=\Omega_{X,\pp}^1$ and the second claim follows.
\end{proof}


The following fact is well-known (see \cite[(2.14)]{Sch16}); we only sketch a proof.

\begin{lem}\label{92}
The modules of regular differential $p$-forms on $X$ and $C$ are related by $\omega_X^p=\Hom_{\OO_C}(\OO_X,\omega_C^p)\subseteq\omega_C^p$.
\end{lem}

\begin{proof}
Kersken explicitly describes (see \cite[(1.2)]{Ker84})
\begin{equation}\label{122}
\omega_X^p=\set{\begin{bmatrix}\xi\\\ul h\end{bmatrix}\xmid\xi\in\Omega_Y^{p+k},\ \I_X\xi\subseteq\I_C\Omega_Y^{p+k},\ d\I_X\wedge\xi\subseteq\I_C\Omega_Y^{p+k+1}}
\end{equation}
where $\begin{bmatrix}\xi\\\ul h\end{bmatrix}=0$ if and only if $\xi\in\I_C\Omega_Y^{p+k}$.
In particular, $\omega_X^p\subseteq\Hom_{\OO_C}(\OO_X,\omega_C^p)\subseteq\omega_C^p$ and equality in $\omega_C^p$ can be checked at $\Ass(\OO_C)$.
Lemma~\ref{93} yields the claim.
\end{proof}


The following modules of differential forms on $Y$ due to Aleksandrov (see \cite[Def.~10.1]{Ale14} and \cite[Def.~4.1.3]{Pol16}) are defined by the relations in \eqref{122}.

\begin{dfn}\label{96}
The module of \emph{multi-logarithmic differential $q$-forms on $Y$ along $X$ relative to $C$} is defined by
\begin{align*}
\Omega^q(\log X/C)=\Omega_Y^q(\log X/C)&:=\set{\omega\in \Omega^q_Y\xmid\I_X\omega\subseteq\I_C\Omega^q_Y,\ d\I_X\wedge\omega\subseteq\I_C\Omega^{q+1}_Y}(D)\\
&=\set{\omega\in\Omega^q_Y(D)\xmid\I_X\omega\subseteq\Sigma\Omega^q_Y,\ d\I_X\wedge\omega\subseteq\Sigma\Omega^{q+1}_Y}.
\end{align*}
Observe that
\[
\Sigma\Omega_Y^q\subseteq\Omega^q(\log X/C)\subseteq\Omega^q(\log C)
\]
with $\Omega^q(\log X/C)(-D)\subseteq\QQ_Y\otimes_{\OO_Y}\Omega_Y^q$ independent of $D$ (see \cite[Prop.~4.1.5]{Pol16}).
\end{dfn}


\begin{lem}\label{76}
There is an equality $\Omega^q(\log X/C)=\Sigma\Omega^q_Y:_{\Omega^q(\log C)}\I_X$.
In other words, $\Omega^q(\log X/C)(-D)=\I_X\Omega^q_Y:_{\Omega^q(\log C)}\I_X$.
\end{lem}

\begin{proof}
There are obvious inclusions
\[
\Sigma\Omega^q_Y\subseteq \Omega^q(\log X/C)\subseteq\Sigma\Omega^q_Y:_{\Omega^q(\log C)}\I_X\subseteq \Omega^q(\log C).
\]
By Proposition~\ref{53} and Lemma~\ref{93}
\begin{align*}
\omega\in\Sigma\Omega^q_Y:_{\Omega^q(\log C)}\I_X
&\implies\I_X\res_C^q(\omega)\subseteq\res_C^q(\Sigma\Omega^q_Y)=0\\
&\implies\res_C^q(\omega)\in\QQ_X\otimes_{\OO_X}\Omega_X^{q-k}\\
&\implies0=d\I_X\wedge\res_C^q(\omega)=\res_C^{q+1}(d\I_X\wedge\omega)\\
&\implies d\I_X\wedge\omega\subseteq \Sigma\Omega^{q+1}_Y\\
&\implies\omega\in\Omega^q(\log X/C).\qedhere
\end{align*}
\end{proof}


The idea of Remark~\ref{59} is used by Aleksandrov (see \cite[\S10]{Ale14}) to define multi-logarithmic residues along $X$ as the restriction of those along $C$.
The bottom sequence of the diagram in the following Proposition~\ref{95} appears in his work (see \cite[Thm.~10.2]{Ale14}); Pol proved exactness on the right (see \cite[Prop.~4.1.21]{Pol16}).
An alternative argument is suggested by \S\ref{65}.
The following data
\begin{equation}\label{113}
R:=\OO_Y,\quad I:=\I_C,\quad J:=\I_X,\quad F:=\Omega_Y^q,\quad M:=\Omega^q(\log C)(-D),\quad\rho:=\frac 1h\res_C^q
\end{equation}
give rise to an $I$-free approximation~\eqref{5} with $J$-restriction~\eqref{108}.
By Corollary~\ref{125} $W=0$ if $q<k$ and \eqref{5} is trivial for $q=n$.
We are therefore concerned with the case $k\le q<n$.
By Lemmas~\ref{92} and \ref{76} (see Definition~\ref{39} and \eqref{81})
\begin{equation}\label{114}
W_T=\omega_X^{q-k},\quad M_J=\Omega^q(\log X/C)(-D).
\end{equation}
Now twisting diagram~\eqref{40} by $D$ yields the following result.

\begin{prp}\label{95}
Applying $\Ext_{\OO_Y}^1(\omega_X^{q-k}\into\omega_C^{q-k},\Sigma\Omega^q_Y)$ to the multi-logarithmic residue sequence~\eqref{123} yields a commutative diagram with exact rows and cartesian right square
\begin{equation}\label{110}
\xymatrix{
0\ar[r] & \Sigma\Omega^q_Y\ar[r]\ar@{=}[d] & \Omega^q(\log C)\ar[r]^-{\res_C^q} & \omega_C^{q-k}\ar[r] & 0\\
0\ar[r] & \Sigma\Omega^q_Y\ar[r] & \Omega^q(\log X/C)\ar@{^{(}->}[u]\ar[r]^-{\res_{X/C}^q} & \omega_X^{q-k}\ar@{^{(}->}[u]\ar[r] & 0
}
\end{equation}
where $\omega_X^p$ is the module of regular meromorphic $p$-forms on $X$.\qed
\end{prp}

\subsection{Higher log vector fields and Jacobian modules}\label{97}

Pol gives a description of $\res_{X/C}^q$ preserving the analogy with the definition of $\res_C^q$ in Proposition~\ref{53} (see \cite[\S4.2.1]{Pol16}).
As suggested by Remark~\ref{59} the role of $\frac{dh}h\in\Omega^k(\log C)$ is played by a preimage $\frac{\alpha_X}h\in\Omega^k(\log X/C)$ of the fundamental form $\begin{bmatrix}\alpha_X\\\ul h\end{bmatrix}\in\omega_X^0$ of $X$ (see \cite[\S5]{Ker83b}).

\begin{dfn}\label{100}
Let $\textbf{1}_X:=(1,0)\in\QQ_X\times\QQ_{X'}=\QQ_C$ (see Lemma~\ref{93}).
A \emph{fundamental form of $X$ in $Y$} is an $\alpha_X=\alpha_{X/C/Y}\in\Omega^k_Y$ such that $\ol{\alpha_X}=\ol{\textbf{1}_Xd\ul h}\in\QQ_C\otimes_{\OO_Y}\Omega_Y^k$.
\end{dfn}


Such a fundamental form exists and the explicit description of multi-logarithmic differential forms in Proposition~\ref{53} generalizes verbatim (see \cite[Prop.~4.2.6]{Pol16}).

\begin{prp}\label{87}
An element $\omega\in\Omega^q_Y(D)$ lies in $\Omega^q(\log X/C)$ if and only if there exist $g\in \OO_Y$ inducing a non zero-divisor in $\OO_C$, $\xi\in \Omega^{q-k}_Y$ and $\eta\in\Sigma\Omega^q_Y$ such that
\[
g\omega=\frac{\alpha_X}{h}\wedge \xi+\eta
\]
and the map $\res_{X/C}^q$ in \eqref{110} is defined by $\res_{X/C}^q(\omega)=\frac\xi g$.\qed
\end{prp}


In the same spirit we extend Definition~\ref{54}.
We start with the first option as definition.

\begin{dfn}\label{101}
We define the module of \emph{multi-logarithmic $q$-vector fields on $Y$ along $X$} by
\[
\Der^q(-\log X)=\Der_Y^q(-\log X):=\set{\delta\in\Theta^q_Y\xmid\ideal{\delta,\wedge^kd\I_X\wedge\Omega^{q-k}_Y}\subseteq\I_X}.
\]
\end{dfn}


The following result completes the analogy with Definition~\ref{54}.
In particular $\Der^k(-\log X)$ is Pol's module $\Der^k(-\log X/C)$ (see \cite[Def.~4.2.8]{Pol16}) which is thus independent of $C$.

\begin{lem}\label{102}
We have
\begin{align*}
\Der^q(-\log C)&\subseteq\set{\delta \in \Theta^q_Y\xmid\ideal{\delta,\alpha_X\wedge\Omega^{q-k}_Y}\subseteq\I_X}
=\Der^q(-\log X)\\
&=\set{\delta \in \Theta^q_Y\xmid\ideal{\delta,\alpha_X\wedge\Omega^{q-k}_Y}\subseteq\I_C}.
\end{align*}
\end{lem}

\begin{proof}
By Definition~\ref{100} $\ol{\alpha_X}=\ol{\textbf{1}_Xd\ul h}=\ol{d\ul h}\in\QQ_X\otimes_{\OO_Y}\Omega^k_Y$.
For $\delta\in\Theta_Y^q$ and $\xi\in\Omega^{q-k}_Y$
\[
\ideal{\delta,\alpha_X\wedge\xi}\in\I_X
\iff0
=\ol{\ideal{\delta,\alpha_X\wedge\xi}}
=\ideal{\ol\delta,\ol{\alpha_X}\wedge\ol\xi}
=\ideal{\ol\delta,\ol{d\ul h}\wedge\ol\xi}
=\ol{\ideal{\delta,d\ul h\wedge\xi}}\in\QQ_X
\]
where $\ol\delta\in\QQ_X\otimes_{\OO_Y}\Theta_Y^q$ and $\ol\xi\in\QQ_X\otimes_{\OO_Y}\Omega^{q-k}_Y$.
The claimed inclusion follows.
Using the Leibniz rule and that $\QQ_Xd\I_C=\QQ_Xd\I_X\subseteq\QQ_X\otimes_{\OO_Y}\Omega_Y^1$ by Lemma~\ref{93} 
\begin{align*}
0=\ideal{\ol\delta,\ol{d\ul h}\wedge\ol\xi}\in\QQ_X
&\iff0
=\ideal{\ol\delta,\wedge^k\ol{d\I_C}\wedge\ol\xi}=\ideal{\ol\delta,\wedge^k\ol{d\I_X}\wedge\ol\xi}=\ol{\ideal{\delta,\wedge^kd\I_X\wedge\xi}}\subseteq\QQ_X\\
&\iff\ideal{\delta,\wedge^kd\I_X\wedge\xi}\subseteq\I_X.
\end{align*}
This proves the first equality.
With $\I_C=\I_X\cap\I_{X'}$ the second equality follows from $\alpha_X\in\I_{X'}\Omega_Y^k$ (see \cite[Prop. 4.2.5]{Pol16}).
\end{proof}


Using Proposition~\ref{87} and Lemma~\ref{102} we obtain the following analogue of Lemma~\ref{55} and of the equality $\Der^q(-\log C)=\Omega^q(\log C)^\Sigma$ from Proposition~\ref{77}.

\begin{lem}\label{103}
For $\delta\in \Der^q(-\log X)$ and $\omega\in \Omega^q(\log X/C)$ we have $\ideal{\delta,\omega}\in\Sigma$.\qed
\end{lem}


\begin{lem}\label{90}
There is an equality $\Der^q(-\log X)=\Omega^q(\log X/C)^\Sigma$.\qed
\end{lem}


The following proposition extends Proposition~\ref{77} and includes the counterpart of Lemma~\ref{98}.

\begin{prp}\label{94}
There are chains of $\OO_Y$-submodules of $\QQ_Y\otimes_{\OO_Y}\Omega_Y^q$ and $\QQ_Y\otimes_{\OO_Y}\Theta_Y^q$ 	
\begin{gather*}
\Omega^q_Y\subseteq\Sigma\Omega^q_Y\subseteq\Omega^q(\log X/C)\subseteq\Omega^q(\log C)\subseteq\Omega^q_Y(D)\subseteq \Sigma\Omega^q_Y(D),\\
\Sigma\Theta^q_Y\supseteq\Theta^q_Y\supseteq\Der^q(-\log X)\supseteq\Der^q(-\log C)\supseteq\I_C\Theta^q_Y\supseteq\Theta^q_Y(-D)
\end{gather*}
that are $\Sigma$-duals of each other.
\end{prp}

\begin{proof}
By Lemma~\ref{86} and Proposition~\ref{77} $M$ in \eqref{113} is $I$-reflexive.
By Proposition~\ref{72} and \eqref{114} $\Omega^q(\log X/C)(-D)$ is therefore $\I_C$-reflexive and, again by Lemma~\ref{86}, $\Omega^q(\log X/C)$ $\Sigma$-reflexive.
The claim follows with Proposition~\ref{77} and Lemmas~\ref{102} and \ref{90}.
\end{proof}


\begin{dfn}
Contraction with $\alpha_X$ defines a map
\[
\alpha^X\colon\Theta_Y^q\to\OO_X\otimes_{\OO_Y}\Theta_Y^{q-k}=\Hom_{\OO_Y}(\Omega_Y^{q-k},\OO_X),\quad\delta\mapsto(\omega\mapsto\ol{\ideal{\delta,\alpha_X\wedge\omega}}).
\]
Taking $p+q=n$ we define the \emph{$p$th Jacobian module of $X$} as the $\OO_X$-module
\[
\J_X^p:=\alpha^X(\Theta_Y^{q}).
\]
\end{dfn}

The Jacobian module $\J_X^{\dim X}$ agrees with Pol's Jacobian ideal $\J_{X/C}$ (see \cite[Not.~4.2.14]{Pol16}) which coincides with the $\omega$-Jacobian ideal if $X$ is Gorenstein (see \cite[Prop.~4.2.34]{Pol16}).


\begin{rmk}\label{106}
In explicit terms
\[
\alpha^X\colon\Theta^q_Y\to\bigoplus_{\ul i\in N^{q-k}_<}\OO_Xdx_{\ul i},\quad
\delta\mapsto\sum_{\ul i\in N^{q-k}_<}\ideal{\delta,\alpha_X\wedge dx_{\ul i}}dx_{\ul i}.
\]
In case $X=C$, $\alpha_C=d\ul h$ and  
\[
\ideal{\delta,d\ul h\wedge dx_{\ul i}}
=\sum_{\ul j\in N^q_<}
\frac{\partial(\ul h,x_{\ul i})}{\partial x_{\ul j}}\ideal{\delta,dx_{\ul j}}.
\]
In particular, $\J_C^{\dim C}$ is the Jacobian ideal of $C$.
\end{rmk}


\begin{lem}\label{124}
If $k\le q\le n$, then $\omega_X^{q-k}\ne0$ and, unless $q=n$, $\OO_X\otimes\alpha^X$ is not injective.
\end{lem}

\begin{proof}
This can be checked at smooth points of $X=C$ where $\ul h=(x_1,\dots,x_k)$ and $\alpha_X=d\ul h$.
Here $\omega_X^{q-k}=\Omega_X^{q-k}\ne0$ and $0\ne\frac\partial{\partial x_{\ul j}}\in\ker(\OO_X\otimes\alpha^X)$ if $\set{1,\dots,k}\not\subseteq\set{j_1,\dots,j_q}$.
\end{proof}


By Lemma~\ref{102} there is a short exact sequence (see \cite[Prop.~4.2.16]{Pol16} for $q=k$)
\begin{equation}\label{112}
\xymatrix{
0 & \J_X^{n-q}\ar[l] & \Theta^q_Y\ar[l]_-{\alpha^X} & \Der_Y^q(-\log X)\ar[l] & 0.\ar[l]
}
\end{equation}


\begin{lem}\label{107}
There is a pairing
\[
\J_X^{n-q}\otimes\omega_X^{q-k}\to\Hom_{\OO_C}(\OO_X,\OO_C)(D)=\omega_X,\quad
\left(\alpha^X(\delta),\res_{X/C}^q(\omega)\right)\mapsto\ideal{\delta,\omega}.
\]
\end{lem}

\begin{proof}
By Lemma~\ref{103} the pairing $\Omega_Y^q(D)\times\Theta_Y^q\to\OO_Y(D)$ obtained from \eqref{115} maps both $\Omega_Y^q(\log X/C)\times\Der_Y^q(-\log X)$ and $\Sigma\Omega_Y^q\otimes\Theta_Y^q$ to $\Sigma$.
Using the bottom row of \eqref{110} and \eqref{112} this yields a pairing $\J_X^{n-q}\otimes\omega_X^{q-k}\to\OO_Y(D)/\Sigma=\OO_C(D)=\omega_C$.
Both $\J_X^{n-q}$ and $\omega_X^{q-k}$ are supported on $X$ and applying $\Hom_{\OO_C}(\OO_X,-)$ yields the claim (see \eqref{116}).
\end{proof}

\medskip
We can now prove our main application.

\begin{proof}[Proof of the Theorem~\ref{0}]
By Lemmas~\ref{86} and \ref{90} sequence \eqref{112} in terms of \eqref{113} is the $I$-dual $J$ restriction~\eqref{109} twisted by $D$, that is, $V^T=\J_X^{n-q}$ and $\alpha^T=\alpha^X$ up to a twist by $D$.
With \eqref{114} and Lemma~\ref{124} the claim now reduces to Corollary~\ref{61}.
The identifications are induced by the pairing in Lemma~\ref{107}.
\end{proof}


\begin{prp}\label{104}
The $\OO_X$-modules $\J_X^{n-q}$ depend only on $X$.
\end{prp}

\begin{proof}
We identify $\J_X^{n-q}=\Theta^q_Y/\Der_Y^q(-\log X)$ by the exact sequence~\eqref{112}.
Any isomorphism $Y'\cong Y$ of minimal embeddings of $X$ induces an isomorphism $\varphi\colon\OO_Y\cong\OO_{Y'}$ over $\OO_X$ identifying $\I_{X/Y}\cong\I_{X/Y'}$.
There are induced compatible isomorphisms $\Theta_Y^q\cong\Theta_{Y'}^q$ and $\Omega_Y^p\cong\Omega_{Y'}^p$ over $\varphi$ resulting in an isomorphism over $\varphi$
\[
\Der_Y^q(-\log X)\cong\Der_{Y'}^q(-\log X).
\]

Any general embedding $X\subseteq Y'$ arises from a minimal embedding $X\subseteq Y$ up to isomorphism of the latter as $Y'=Y\times Z$ where $Z\cong(\CC^m,0)$ and hence
\[
\I_{X/Y'}=\OO_Y\hat\otimes\mm_Z+\I_{X/Y}\hat\otimes\OO_Z.
\]
Pick coordinates $z_1,\dots,z_m$ on $Z$ and abbreviate $d\ul z:=dz_1\wedge\dots\wedge dz_m$ and $\frac\partial{\partial\ul z}:=\frac\partial{\partial z_1}\wedge\dots\wedge\frac\partial{\partial z_m}$.
Then there are decompositions
\[
\Omega_{Y'}^{q+m}=\OO_{Z}\hat\otimes\Omega_Y^q\wedge d\ul z\oplus\wt\Omega_{Y'}^{q+m},\quad
\Theta_{Y'}^{q+m}=\OO_{Z}\hat\otimes\Theta_Y^q\wedge\frac\partial{\partial\ul z}\oplus\wt\Theta_{Y'}^{q+m}
\]
where the modules with tilde are generated by basis elements not involving $d\ul z$ and $\frac\partial{\partial\ul z}$ respectively.
Fundamental forms of $X$ in $Y'$ and $Y$ can be chosen compatibly as
\[
\alpha_{X/C/Y'}=\alpha_{X/C/Y}\wedge d\ul z\in\Omega_{Y'}^{k+m}.
\]
With Lemma~\ref{102} this yields inclusions
\[
\Der_Y^q(-\log X)\wedge\frac\partial{\partial\ul z}+\wt\Theta_{Y'}^{q+m}\subseteq\Der_{Y'}^{q+m}(-\log X)\supseteq\I_{X/Y'}\Theta_{Y'}^{q+m}\supseteq\mm_Z\hat\otimes\Theta_{Y}^q\wedge\frac\partial{\partial\ul z}
\]
and a cartesian square
\[
\xymatrix{
\OO_Z\hat\otimes\Theta_Y^q\ar@{^(->}[r]^-{-\wedge\frac\partial{\partial\ul z}} & \Theta_{Y'}^{q+m}\\
\Der_Y^q(-\log X)+\mm_Z\hat\otimes\Theta_{Y}^q\ar@{^(->}[r]\ar@{^(->}[u] & \Der_{Y'}^{q+m}(-\log X).\ar@{^(->}[u]
}
\]
It gives rise to an isomorphism of $\OO_X$-modules
\begin{align*}
\Theta_{Y'}^{q+m}/\Der_{Y'}^{q+m}(-\log X)
&\cong\OO_Z\hat\otimes\Theta_Y^q/(\Der_Y^q(-\log X)+\mm_Z\hat\otimes\Theta_{Y}^q)\\
&\cong\Theta_Y^q/\Der_Y^q(-\log X).\qedhere
\end{align*}
\end{proof}

\bibliographystyle{amsalpha}
\bibliography{residual-arxiv}
\end{document}